\documentclass{elsarticle}

\usepackage{amsmath,amsfonts,amssymb,amsthm}
\usepackage{enumitem}
\usepackage{geometry}[margin=2cm]
\usepackage{mathabx} 
\usepackage{hyperref}

\newtheorem{theorem}{Theorem}[section]
\newdefinition{definition}[theorem]{Definition}
\newtheorem{proposition}[theorem]{Proposition}
\newtheorem{corollary}[theorem]{Corollary}
\newtheorem{lemma}[theorem]{Lemma}
\newtheorem{remark}[theorem]{Remark}
\allowdisplaybreaks

\begin{document}

\title{\bf A functional-analytic construction of the stochastic parallel transport in Hermitian bundles over Riemannian manifolds}
\author{Alexandru Mustăţea}
\ead{Alexandru.Mustatea@imar.ro}
\address{"Simion Stoilow" Institute of Mathematics of the Romanian Academy, P.O. Box 1-764, Bucharest, RO-014700, Romania}
\date{}

\begin{abstract}
This article presents a purely functional-analytic construction of the concept of stochastic parallel transport in Hermitian bundles over Riemannian manifolds. As a byproduct, we also obtain a form of the Feynman-Kac formula in vector bundles that is, to our best knowledge, the most general found so far.
\end{abstract}

\maketitle

\section{Motivation and outline of this article}

The concept of "stochastic parallel transport" in a vector bundle $E$ over a Riemannian manifold $M$ is usually presented as a byproduct of the concept of "stochastic differential equation"; this is the approach taken in most texts, for instance in \cite{IW89} and in \cite{Meyer82}. Nevertheless, K. It\^{o} had originally conceived it differently (\cite{Ito63}, \cite{Ito75a}, \cite{Ito75b}): for every continuous curve $c : [0,t] \to M$, consider the unique geodesic segment joining the consecutive "dyadic" points $c (\frac {jt} {2^k})$ and $c (\frac {(j+1)t} {2^k})$, join these $2^k$ geodesic segments into a single zig-zag piecewise-geodesic line, and parallel-transport the vector $v \in E_{c(0)}$ along this line to $E_{c(t)}$; for Wiener-almost all continuos curves $c$, the limit when $k \to \infty$ will exist and will be called "the stochastic parallel transport of $v$ along $c$". Both approaches are equivalent, as shown in \cite{Meyer82} and \cite{Emery90}, and both are constructed within the framework of probability theory, therefore being accessible mostly to probabilists. The aim of this article is to reconstruct the concept of "stochastic parallel transport" using only functional-analytic tools and concepts, thus opening it up to a much larger class of mathematicians.

Since the constructions in this text will be fairly technical, let us sketch the intuition underpinning them. Let $D_t = \{ \frac {jt} {2^k} \mid k \in \mathbb N, \ j \in \mathbb N \cap [0,2^k] \}$ - the "dyadic" numbers between $0$ and $t$. Following It\^{o}'s idea, the parallel transport of $v \in E_{c(0)}$ along the zig-zag line determined by the points $\{ c(0), c(\frac t {2^k}), \dots, c (\frac {(2^k-1)t} {2^k}), c(t) \}$ is the parallel transport $T_{k,0}$ from $c(0)$ to $c(\frac t {2^k})$, followed by the parallel transport $T_{k,1}$ from $c(\frac t {2^k})$ to $c(\frac {2t} {2^k})$ and so on, ending with the parallel transport $T_{k, 2^k-1}$ from $c(\frac {(2^k-1)t} {2^k})$ to $c(t)$; symbolically, it is $T_{k, 2^k-1} \dots T_{k,0} v$. Now comes the remark that is the backbone of the present work: $T_{k, 2^k-1} \dots T_{k,0} v$ can be viewed as the "contraction" of all the tensor products in
\begin{align*}
T_{k, 2^k-1} \otimes \dots \otimes T_{k,0} \otimes v & \in \left( E_{c(t)} \otimes E_{c(\frac {(2^k-1)t} {2^k})} ^* \right) \otimes \dots \otimes \left( E_{c(\frac t {2^k})} \otimes E_{c(0)} ^* \right) \otimes E_{c(0)} \simeq \\
& \simeq E_{c(t)} \otimes \left( E_{c(\frac {(2^k-1)t} {2^k})} ^* \otimes E_{c(\frac {(2^k-1)t} {2^k})} \right) \otimes \dots \otimes \left( E_{c(0)} ^* \otimes E_{c(0)} \right) \simeq \\
& \simeq E_{c(t)} \otimes \operatorname{End} E_{c(\frac {(2^k-1)t} {2^k})} ^* \otimes \dots \otimes \operatorname{End} E_{c(0)} ^* \ .
\end{align*}
Let us see now what "contraction" means. If $U_1, \dots, U_N$ are finite-dimensional vector spaces, if $u \in U_N$ and $\omega \in U_1 ^*$, and $A_j : U_{j+1} \to U_j$ is a linear operator for all $1 \le j \le N-1$, then
\[ \omega \otimes A_1 \otimes \dots A_{N-1} \otimes u \in U_1 ^* \otimes (U_1 \otimes U_2 ^*) \otimes \dots \otimes (U_{N-1} \otimes U_N ^*) \otimes U_N \simeq \operatorname{End} U_1^* \otimes \dots \otimes \operatorname{End} U_N ^* \ ; \]
if $\operatorname{Id} _{U_j}$ is the identity operator on $U_j$, then $\operatorname{Id} _{U_1} \otimes \dots \otimes \operatorname{Id} _{U_N} \in \operatorname{End} U_1 \otimes \dots \otimes \operatorname{End} U_N$, therefore it makes sense to apply $\omega \otimes A_1 \otimes \dots \otimes A_N \otimes u$ on $\operatorname{Id} _{U_1} \otimes \dots \otimes \operatorname{Id} _{U_N}$, the result being $\omega (A_1 \dots A_N u)$. We see that in order to perform this contraction in the product of parallel transports considered above we need to add in a supplementary factor $E_{c(t)} ^*$ with which to pair the factor $E_{c(t)}$ in order to obtain $\operatorname{End} E_{c(t)} ^*$ and be able to perform the contraction described above. This means that if $\eta_{c(t)} \in E_{c(t)}$, then
\[ \eta_{c(t)} \otimes T_{k, 2^k-1} \otimes \dots \otimes T_{k,0} \otimes v \in \operatorname{End} E_{c(t)} ^* \otimes \dots \otimes \operatorname{End} E_{c(0)} ^* \]
and
\[ \eta_{c(t)} (T_{k, 2^k-1} \dots T_{k,0} v) = (\eta_{c(t)} \otimes T_{k, 2^k-1} \otimes \dots \otimes T_{k,0} \otimes v) (\operatorname{Id} _{E_{c(t)}} \otimes \dots \otimes \operatorname{Id} _{E_{c(0)}}) \ .\]
Following now in the footsteps of Itô, we let $k \to \infty$; what we get, then, will be a contraction between tensor products with infinitely many factors; the rigorous construction of these tensor products will be our first task, but we can say that these tensor product spaces will be $\mathcal E _c = \otimes _{s \in D_t} \operatorname{End} E_{c(s)}$ and its dual. If we denote the space of continuous curves by $\mathcal C_t$, the fact that $\mathcal E_c$ depends on $c \in \mathcal C_t$ suggests that the disjoint union $\coprod _{c \in \mathcal C_t} \mathcal E _c$ will be a (topological) vector bundle of infinite rank over $\mathcal C_t$. Since $\eta_{c(t)} \otimes T_{k, 2^k-1} \otimes \dots \otimes T_{k,0} \otimes v$ takes values in the fiber $\mathcal E_c ^*$ for all $k \in \mathbb N$ and all $c \in \mathcal C_t$, we deduce that these tensor products will all be some kind of sections in $\mathcal E ^*$, whence it is reasonable to assume that their limit for $k \to \infty$ (the stochastic parallel transport, once we get rid of $\eta$) will be a section of the same kind. Indeed, this will turn out to be the case, and in order to obtain this we shall resort to Chernoff's theorem about the approximation of contraction semigroups.

An unexpected byproduct of the construction in this article is a new version of the Feynman-Kac formula in vector bundles: not only will its proof be completely new, but its hypotheses seem to be the most general considered so far in the literature, to the author's best knowledge; more precisely, the potential will be taken to be only locally-integrable and lower-bounded, while no restrictions will be imposed upon the manifold.

The plan of the article is the following, the notations going to be explained as soon as they become necessary:
\begin{itemize}[wide]
\item we shall construct a Hermitian vector bundle $\mathcal E$ over $\mathcal C_t$, the fibers of which will be infinite-dimensional Hilbert spaces;
\item we shall consider spaces of square-integrable sections in $\mathcal E$ and $\mathcal E ^*$ and, in particular, we shall obtain by an abstract argument a specific essentially bounded section $\rho_{t, \omega, \eta}$, which will be the limit of a sequence of sections $(P_{t, \omega, \eta, k}) _{k \in \mathbb N}$ given by explicit formulae;
\item we shall emphasize a conjugate-linear continuous map $\mathcal P _{t,v} ^2 : \Gamma^2 (\mathcal E) \to \Gamma^2 (p_t ^* E)$, which we shall see to enclose a lot of information about both the geometry of the bundle $E \to M$ and the Wiener measure $w_t$ on $\mathcal C_t$;
\item using the map $\mathcal P _{t,v} ^2$ we shall be able to give meaning to the concept of stochastic parallel transport from a functional-analytic point of view;
\item finally, using the same map $\mathcal P _{t,v} ^2$, we shall study an extension of the Feynman-Kac formula in the bundle $E$.
\end{itemize}

Finally, it is a pleasure to thank Mr. Radu Purice of the "Simion Stoilow" Institute of Mathematics of the Romanian Academy for his constant moral and mathematical support offered during the elaboration of the present work. His seemingly infinite patience in reading the successive draft versions of this article has helped finding and eliminating many errors, and the discussions with him helped me understand the correct mathematical setting in which to place the problem discussed here such that its solution ended up emerging naturally.

\section{A Hermitian vector bundle of infinite rank}

In the following, $M$ will be a separable connected Riemannian manifold of dimension $n$, and $x_0 \in M$ some fixed arbitrary point. We shall denote by $d : M \times M \to [0, \infty)$ the distance induced on $M$ by the Riemannian structure.

If $t>0$, we shall repeatedly make use of the space
\[ \mathcal C _t = \{ c:[0,t] \to M \mid c \text{ is continuous, with } c(0) = x_0 \} \ , \]
that we shall endow with the topology given by the distance $D(c_1, c_2) = \max _{s \in [0,t]} d(c_1(s), c_2(s))$ and with the natural Wiener measure $w_t$ (a non-probabilistic, functional-analytic and geometric construction of the latter may be found in \cite{BP11}). It is known that $\mathcal C_t$ endowed with this topology is separable (see \cite{Michael61}).

Since we shall be working with various Banach or Hilbert spaces, the norm and the Hermitian product on each of them will be displayed as a lower index: if $v, w \in X$, then $\| v \| _X$ will be the norm of $v$ and $\langle v, w \rangle _X$ will be the Hermitian product of $v$ and $w$. For bounded linear operators between normed spaces, $\| \cdot \|_{op}$ will denote the operator norm, without us specifying the spaces when they are clear from the context.

Let $E \to M$ be a Hermitian vector bundle of complex rank $r \in \mathbb N$, endowed with a Hermitian connection $\nabla$. The fiber of $E$ over $x \in M$ will be denoted by $E_x$, and the Hermitian product on it will be $\langle \cdot, \cdot \rangle _{E_x}$ (all the Hermitian products used in this text will be linear in the first argument).

Let $D_t = \{ \frac {jt} {2^k} \mid k \in \mathbb N, \ j \in \mathbb N \cap [0,2^k] \}$ - the "dyadic" numbers between $0$ and $t$. Our purpose in this section is to give meaning to the Hermitian vector bundle described intuitively by $\mathcal E = \bigboxtimes _{s \in D_t} \operatorname{End} E \to \mathcal C_t$. If $c \in \mathcal C_t$ we let the fiber $\mathcal E _c$ of $\mathcal E$ over $c$ be $\bigotimes _{s \in D_t} \operatorname{End} E_{c(s)}$. This is a tensor product of countably many factors, the definition of which is not trivial and deserves some clarifications. As such, we endow the space $\operatorname{End} E_x$ with the Hermitian product given by $\langle A, B \rangle _{\operatorname{End} E_x} = \frac 1 r \operatorname{Trace} (A B^*)$ for $A, B \in \operatorname{End} E_x$, for every $x \in M$. Notice that $\langle \cdot , - \rangle _{\operatorname{End} E_x} = \frac 1 r \langle \cdot , - \rangle _{E_x \otimes E_x ^*}$, the Hermitian product on the right-hand side being the natural one on $E_x \otimes E_x ^*$. If $\operatorname{Id} _{E_x} \in \operatorname{End} E_x$ is the identity operator, then $\| \operatorname{Id} _{E_x} \| _{\operatorname{End} E_x} = 1$. This allows us to construct the tensor product $\mathcal E _c$ rigorously as follows. If $D_{t,k} = \{ \frac {jt} {2^k} \mid j \in \mathbb N \cap [0,2^k] \}$ for each $k \in \mathbb N$, then for every $k \le k'$ we identify the tensor monomial $\otimes _{s \in D_{t,k}} e_{c(s)} \in \bigotimes _{s \in D_{t,k}} \operatorname{End} E_{c(s)}$ with the monomial $\otimes _{s \in D_{t,k'}} e'_{c(s)} \in \bigotimes _{s \in D_{t,k'}} \operatorname{End} E_{c(s)}$ in which $e'_{c(s)} = e_{c(s)}$ for $s \in D_{t,k}$ and $e'_{c(s)} = \operatorname{Id} _{E_{c(s)}}$ for $s \in D_{t,k'} \setminus D_{t,k}$. This procedure identifies the space $\bigotimes _{s \in D_{t,k}} \operatorname{End} E_{c(s)}$ with a subspace of $\bigotimes _{s \in D_{t,k'}} \operatorname{End} E_{c(s)}$, which allows us to consider the algebraic inductive limit $\varinjlim _{k \in \mathbb N} \bigotimes _{s \in D_{t,k}} \operatorname{End} E_{c(s)}$. This being the algebraic inductive limit of finite tensor products of finite-dimensional Hilbert spaces, it will carry a natural Hermitian product; one then considers the Hilbert space completion of the algebraic inductive limit with respect to this Hermitian product, the resulting Hilbert space being denoted $\mathcal E _c$. It is important to notice that $\mathcal E _c$ is separable because the index set in the inductive limit is $\mathbb N$ and each space in the inductive limit is finite-dimensional.

We define now the total space of the putative Hermitian vector bundle as $\mathcal E = \bigcup _{c \in \mathcal C_t} \{ c \} \times \mathcal E _c$. The natural projection of $\mathcal E$ onto $\mathcal C_t$ will be $\operatorname{pr}_{\mathcal E} : \mathcal E \to \mathcal C_t$, given by the projection $\operatorname{pr}_{\mathcal E} ((c,e)) = c$. So far, $\mathcal E$ has been constructed fiberwise only as a set; in what follows, we shall endow it with a topology, but in order to do this we shall need an auxiliary result.

\begin{lemma} \label{approximation of continuous curves}
For every continuous curve $c \in \mathcal C_t$ and every $\varepsilon > 0$ there exists a piecewise-smooth curve $c' : [0,t] \to M$ such that $D(c, c') < \varepsilon$.
\end{lemma}
\begin{proof}
Let $c \in \mathcal C_t$. The idea of the proof is the following: if $k \in \mathbb N$ is large enough, then the points $c(0), c(\frac t {2^k}), \dots, c(\frac {(2^k-1)t} {2^k}), c(t)$ will be close enough to each other so that any two consecutive of them may be joined by a unique minimizing geodesic; by joining these geodesic segments together, we shall obtain a piecewise-smooth curve $c'$ (a geodesic interpolation of the $2^k+1$ points above) which, for sufficiently large $k$, will be at distance at most $\varepsilon$ from $c$. The rest of the proof formalizes this idea rigourously.

Being defined on a compact interval, $c$ will be uniformly continuous; let $\delta$ be an increasing modulus of continuity for it. Since
\[ d \left( c \left( \frac {jt} {2^k} \right), c \left( \frac {(j+1)t} {2^k} \right) \right) \le \delta \left( \frac t {2^k} \right) \to 0 \ , \]
for every $0 \le j \le 2^k - 1$, we deduce that for large enough $k \in \mathbb N$ the points $c(\frac {jt} {2^k})$ and $c(\frac {(j+1)t} {2^k})$ may be joined by a unique minimizing geodesic $\gamma_{k,j} : [0,1] \to M$ with $\gamma_{k,j} (0) = c(\frac {jt} {2^k})$ and $\gamma_{k,j} (1) = c(\frac {(j+1)t} {2^k})$, for all $0 \le j \le 2^k - 1$. Let us consider the piecewise-geodesic curve $c' : [0,t] \to M$ obtained by gluing these geodesic segments together: on evey interval $[\frac {jt} {2^k}, \frac {(j+1)t} {2^k}]$ it will be given by $c' (s) = \gamma_{k,j} (\frac {2^k} t s - j)$, for all $0 \le j \le 2^k - 1$. Using the triangle inequality in the triangle of vertices $c(\frac {jt} {2^k})$, $c(s)$ and $c'(s)$ for $s \in [\frac {jt} {2^k}, \frac {(j+1)t} {2^k}]$, let us notice that
\begin{align*}
D(c, c') & = \max_{0 \le j \le 2^k - 1} \max_{\frac {jt} {2^k} \le s \le \frac {(j+1)t} {2^k}} d \left( c(s), \gamma_{k,j} \left( \frac {2^k} t s - j \right) \right) \le \\
& \le \max_{0 \le j \le 2^k - 1} \max_{\frac {jt} {2^k} \le s \le \frac {(j+1)t} {2^k}} d \left( c(s), c \left( \frac {jt} {2^k} \right) \right) + d \left( c \left( \frac {jt} {2^k} \right), \gamma_{k,j} \left( \frac {2^k} t s - j \right) \right) \le \\
& \le \max_{0 \le j \le 2^k - 1} \max_{\frac {jt} {2^k} \le s \le \frac {(j+1)t} {2^k}} d \left( c(s), c \left( \frac {jt} {2^k} \right) \right) + d \left( c \left( \frac {jt} {2^k} \right), c \left( \frac {(j+1)t} {2^k} \right) \right) \le \\
& \le \max_{0 \le j \le 2^k - 1} \max_{\frac {jt} {2^k} \le s \le \frac {(j+1)t} {2^k}} \delta \left( s - \frac {jt} {2^k} \right) + \delta \left( \frac t {2^k} \right) \le 2 \delta \left( \frac t {2^k} \right) \, 
\end{align*}
whence it follows that $D(c, c') < \varepsilon$ for large enough $k$.
\end{proof}

We shall construct the topology on $\mathcal E$ first locally, on the restrictions of $\mathcal E$ to open balls $B(c,r)$ centered at each curve $c \in \mathcal C_t$, and then we shall show that all these local topologies are compatible with each other, which will allow us to glue them together into a global topology on $\mathcal E$.


For every $r \in (0, \min_{s \in [0,t]} \operatorname{injrad} (c(s)))$ consider the open metric ball $B(c, r) = \{ \gamma \in \mathcal C_t \mid D(c, \gamma) < r \}$ and a piecewise-smooth curve $c' \in B(c, r)$, the existence of which being guaranteed by lemma \ref{approximation of continuous curves}. If $\gamma \in B(c,r)$ and $s \in [0,t]$ then
\[ d(\gamma(s), c(s)) < D(\gamma, c) < \min_{s \in [0,t]} \operatorname{injrad} (c(s)) < \operatorname{injrad} (c(s)) \ , \]
so there exists a unique minimizing geodesic defined on $[0,1]$ from $\gamma(s)$ to $c(s)$. Next, using the same argument, there exists a unique minimizing geodesic defined on $[0,1]$ from $c(s)$ to $c'(s)$. We may then parallel-transport the vector $e \in E_{\gamma(s)}$ to $c(s)$, and then to $c'(s)$, each time along the geodesics found above; we finally parallel-transport the vector obtained so far from $c'(s)$ to $c'(0) = x_0$ along $c'$, which is piecewise-smooth, thus obtaining a vector in $E_{x_0}$. The procedure just described gives a linear isometry from $E_{\gamma(s)}$ to $E_{x_0}$; it is clear that it may be inverted (by traversing the same curves in the opposite direction and in the inverse order), so this procedure is an isometric isomorphism. We may extend it in the natural way to tensor monomials of the form $e_{\gamma(s_1)} \otimes \dots \otimes e_{\gamma(s_N)} \in \mathcal E_\gamma$ with $N \in \mathbb N \setminus \{0\}$ and $s_1, \dots, s_N \in D_t$, thus obtaining tensor monomials in $(\operatorname{End} E_{x_0}) ^{\otimes \{ s_1, \dots, s_N \}} \subset (\operatorname{End} E_{x_0}) ^{\otimes D_t}$. This extension will still be a surjective isometry between monomials.

Let us now introduce two helpful auxiliary notations: if $x,y \in M$ and if $c$ is a piecewise-smooth curve from $x$ to $y$, and if $e \in E_x$, then we shall denote by $PT_{x \to y, c} (e) \in E_y$ the parallel transport of $e$ from $x$ to $y$ along $c$. The unique minimizing geodesic defined on $[0,1]$ from $x$ to $y$ will be denoted by $\gamma_{x,y}$, whenever it exists.

We may now define a local trivialization $\varphi : \operatorname{pr}_{\mathcal E}^{-1} (B(c,r)) \to B(c,r) \times (\operatorname{End} E_{x_0}) ^{\otimes D_t}$ as follows:
\begin{itemize}[wide]
\item if $\alpha \in B(c,r)$ and $e_{\alpha(s_1)} \otimes \dots \otimes e_{\alpha(s_N)} \in \mathcal E_\alpha$, then
\begin{align*}
\varphi ((\alpha, \, & e_{\alpha(s_1)} \otimes \dots \otimes e_{\alpha(s_N)})) = \\
& = (c, TP_{c'(s_1) \to x_0, c'} \, TP_{c(s_1) \to c'(s_1), \gamma_{c(s_1), c'(s_1)}} \, TP_{\alpha(s_1) \to c(s_1), \gamma_{\alpha(s_1), c(s_1)}} e_{\alpha(s_1)} \otimes \dots \\
& \dots \otimes TP_{c'(s_N) \to x_0, c'} \, TP_{c(s_N) \to c'(s_N), \gamma_{c(s_N), c'(s_N)}} \, TP_{\alpha(s_N) \to c(s_N), \gamma_{\alpha(s_N), c(s_N)}} e_{\alpha(s_N)}) \ ,
\end{align*}
as explained above;
\item on linear combinations of such tensor monomials we extend $\varphi$ by linearity, thus obtaining a linear isometric isomorphism;
\item since $\mathcal E_\alpha$ is the Hilbert completion of an algebraic inductive limit, we define $\varphi$ on limits of elements from the algebraic inductive limit by continuity.
\end{itemize}

The map $\varphi : \mathcal E | _{B(c,r)} \to B(c,r) \times (\operatorname{End} E_{x_0}) ^{\otimes D_t}$ allows us now to define a topology on $\mathcal E | _{B(c,r)}$ by transporting the topology from $B(c,r) \times (\operatorname{End} E_{x_0}) ^{\otimes D_t}$ back under $\varphi ^{-1}$. In particular, since $B(c,r)$ is a metric space and $(\operatorname{End} E_{x_0}) ^{\otimes D_t}$ is a Hilbert space, the topology so constructed on $\mathcal E | _{B(c,r)}$ will be first-countable.

It remains to show that these topologies defined only locally are compatible with each other. More precisely, let us show that if the curves $c_1, c_2 \in \mathcal C_t$ and the numbers $r_1, r_2 > 0$ are such that $B(c_1, r_1) \cap B(c_2, r_2) \ne \emptyset$, and if $\varphi_1, \varphi_2$ are two local trivializations above these two balls constructed as above, then the local topologies induced by $\varphi_1$ and $\varphi_2$ coincide on $\mathcal E | _{B(c_1, r_1) \cap B(c_2, r_2)}$. But this is easy, since the map $\phi_1 \circ \phi_2 ^{-1} : B(c_1, r_1) \cap B(c_2, r_2) \times (\operatorname{End} E_{x_0}) ^{\otimes D_t} \to B(c_1, r_1) \cap B(c_2, r_2) \times (\operatorname{End} E_{x_0}) ^{\otimes D_t}$ is the identity on the first factor, and is a continuous map defined on $B(c_1, r_1) \cap B(c_2, r_2)$ with values in the group of isometries of $(\operatorname{End} E_{x_0}) ^{\otimes D_t}$ on the second factor, whence the conclusion is clear.

Since the local topologies constructed above have turned out to be compatible with each other, they may be glued together into a unique (first-countable) global topology on $\mathcal E$. In this topology, the maps $\varphi$ constructed above become continuous local trivializations.

\begin{remark}
Since $\mathcal C_t$ is separable, it follows from the above considerations that $\mathcal C_t$ may be covered by a countable family of trivialization domains, a fact which will be useful later on.
\end{remark}

Let $\pi_k : \mathcal C_t \to M^{2^k + 1}$ denote the projection given by
\[ \pi_k (c) = \left( c(0), c \left( \frac t {2^k} \right), \dots, c\left( \frac {(2^k-1)t} {2^k} \right), c(t) \right) \ . \]

\begin{proposition}
The projections $\pi_k : \mathcal C_t \to M^{2^k+1}$ and the projection $\operatorname{pr}_{\mathcal E} : \mathcal E \to \mathcal C_t$ are continuous, for all $k \in \mathbb N$.
\end{proposition}
\begin{proof}
If we denote by $d_k$ the distance induced by the Riemannian tensor on $M^{2^k+1}$ for every $k \in \mathbb N$, then
\begin{align*}
d_k (\pi_k(c), \pi_k(c')) & = d\left( \left( c(0), c\left( \frac t {2^k} \right), \dots, c(t) \right) , \left( c'(0), c'\left( \frac t {2^k} \right), \dots, c'(t) \right) \right) = \\
& = \sqrt{ \sum_{j=0} ^{2^k} d \left( c\left( \frac {jt} {2^k} \right), c'\left( \frac {jt} {2^k} \right) \right) ^2 } \le 2 ^{\frac k 2} \sup_{0 \le j \le 2^k} d \left( c\left( \frac {jt} {2^k} \right), c'\left( \frac {jt} {2^k} \right) \right) \le \\
& \le 2 ^{\frac k 2} \sup_{0 \le s \le t} d(c(s), c'(s)) = 2 ^{\frac k 2} D(c, c') \ ,
\end{align*}
so $\pi_k$ is continuous.

Since the local trivializations constructed above are continuous, and since the restriction of the projection $\operatorname{pr}_{\mathcal E}$ to such trivializations has the form $(c, \dots) \mapsto c$, the continuity of this projection is clear.
\end{proof}

\section{Integrable sections in bundles of infinite rank}

If $s$ is a section of $E$, the notation $\| s \|$ (without any indices) will denote the function $M \ni x \mapsto \| s(x) \|_{E_x} \in [0, \infty)$. The space $\Gamma_0 (E)$ will be the space of compactly-supported smooth sections in $E$, the space $\Gamma_c (E)$ will be the space of compactly-supported continuous sections in $E$, and $\Gamma_{cb} (E)$ will be the space of bounded continuous sections in $E$ (i.e. those continuous sections $\eta$ such that $\sup _{x \in M} \| \eta_x \| _{E_x} < \infty$). For each $1 \le p \le \infty$ the space $\Gamma^p (E)$ will be the space of classes of measurable sections that coincide almost everywhere, with the property that $\| s \| \in L^p(M)$. It is known that $\Gamma_0 (E)$ is dense in $\Gamma^p (E)$ in the norm topology if $p \ne \infty$, and in the weak-$*$ topology if $p = \infty$. The corresponding spaces of \textit{locally} $p$-integrable sections will be $\Gamma^p _{loc} (E)$.

The quadratic form $Q_{E, \nabla} : \Gamma_0 (E) \subset \Gamma^2 (E) \to \mathbb R$ defined by $Q_{E, \nabla} (\eta) = \int _M \| (\nabla \eta) _x \| _{E_x} ^2 \, \mathrm d x$ gives rise to a self-adjoint, positive, densely-defined operator $H_\nabla$ in $\Gamma^2 (E)$ (the Friedrichs extension of the connection Laplacean); by functional calculus one may define next the contraction semigroup $(\mathrm e ^{-s H_\nabla}) _{s \ge 0}$ acting in $\Gamma^2 (E)$, which we shall call "the heat semigroup in $E$ corresponding to $\nabla$" (full details can be found in \cite{Davies80}). It is shown then in chapter XI of \cite{Guneysu17} that this semigroup admits a unique integral kernel ("the heat kernel in $E$ corresponding to $\nabla$"), that is a jointly measurable map $(0, \infty) \times M \times M \ni (s,x,y) \mapsto h_\nabla (s,x,y) \in E_x \otimes E_y ^* \subset E \boxtimes E^*$ such that $h_\nabla (s, x, \cdot) \in \Gamma^2 (E^*)$, $h_\nabla (s, \cdot, y) \in \Gamma^2 (E)$, and $(\mathrm e ^{-s H_\nabla} \eta) (x) = \int _M h_\nabla (s, x, y) \, \eta (y) \, \mathrm d y$ for almost all $x, y \in M$, all $s>0$ and all $\eta \in \Gamma^2 (E)$. It is proved in the same chapter that $h_\nabla (s,x,y) ^* = h_\nabla (s,y,x)$ for all $s>0$ and almost all $x,y \in M$, where the star denotes the adjoint with respect to the Hermitian products on the fibers $E_x$ and $E_y$. One then shows that $h_\nabla$ satisfies locally the partial differential equation $(2 \partial_s + H_{\nabla,x} + H_{\nabla,y}) u = 0$ in the distributional sense (where $H_{\nabla,x}$ means the operator $H_\nabla$ acting on the argument $x$), whence it follows that $h_\nabla$ is smooth using theorem 1 in \cite{Mizohata57}. The same conclusions hold if instead of working on $M$ we work on some relatively compact open subset of it with smooth boundary. If $M = \bigcup _{i \in \mathbb N} U_i$ is an exhaustion of $M$ with such subsets, we shall use the notation $H_\nabla ^{(i)}$ for the Friedrichs extension of the connection Laplacean acting in $\Gamma^2 (E | _{U_i})$, and the corresponding heat kernel will be $h _\nabla ^{(i)}$. In the special case when the vector bundle is $M \times \mathbb C$ endowed with the usual Hermitian product and with the trivial connection, the Friedrichs extension of the connection Laplacean will be denoted simply by $H$, and the corresponding heat kernel simply by $h$; when working on a domain $U_i$ as above these will be $H^{(i)}$ and, respectively, $h^{(i)}$. It is known that $h_i \to h$ pointwise and monotonically (theorem 4 in chapter VIII of \cite{Chavel84}). It is shown in subchapter VII.3 of \cite{Guneysu17} that $\| h_\nabla (t,x,y) \| _{op} \le h(t,x,y)$ for all $t>0$ and almost all $x,y \in M$; since both these heat kernels have been seen to be smooth, and since co-null subsets are dense in $M$, it follows that the inequality is in fact true for all $x,y \in M$. A similar inequality holds on domains $U_i$ as above. This result is known as the \textbf{"diamagnetic inequality"} and will turn out to be crucial in our construction below.

\begin{definition}
We shall say that the section $\sigma : \mathcal C_t \to \mathcal E$ is a \textbf{cylindrical section} if and only if there exists a section $s \in \Gamma^\infty \left( (\operatorname{End} E)^{\boxtimes (2^k + 1)} \right)$ such that $\sigma = s \circ \pi_k$.
\end{definition}

\begin{definition}
We define the Lebesgue space $\Gamma^2 (\mathcal E)$ of square-integrable sections as the space of measurable sections $\sigma : \mathcal C_t \to \mathcal E$ identified under equality almost everywhere, with the property that the function $\mathcal C_t \ni c \mapsto \| \sigma(c) \| _{\mathcal E_c} \in [0, \infty)$ is in $L^2 (\mathcal C_t, w_t)$.
\end{definition}

\begin{theorem}
The space $\Gamma^2 (\mathcal E)$ endowed with the scalar product
\[ \langle \sigma_1, \sigma_2 \rangle _{\Gamma^2 (\mathcal E)} = \int _{\mathcal C_t} \langle \sigma_1 (c), \sigma_2 (c) \rangle _{\mathcal E_c} \, \mathrm d w_t (c) \]
is a Hilbert space. Its dual is $\Gamma^2 (\mathcal E ^*)$, where $\mathcal E^*$ is the dual bundle of $\mathcal E$ in which the fiber $\mathcal E_c ^*$ is the dual space of $\mathcal E_c$ for all $c \in \mathcal C_t$.
\end{theorem}
\begin{proof}
That $\Gamma^2 (\mathcal E)$ is an inner product space is easy. The proof of its metric completeness follows faithfully the usual one of the completeness of the space $L^2$. The main ingredients are the fact that each fiber is, in turn, complete (being a Hilbert space), and the fact that $\mathcal C_t$ may be covered by a countable family of trivialization domains (a consequence if its separability). More specifically, assume that $(\sigma_k) _{k \in \mathbb N} \subset \Gamma^2 (\mathcal E)$ is a Cauchy sequence. There exists a subsequence $(\sigma_{k_m}) _{m \in \mathbb N}$ such that
\[ \| \sigma_{k_{m+1}} - \sigma_{k_m} \| _{\Gamma^2 (\mathcal E)} \le \frac 1 {2^{m+1}} \]
for each $m \in \mathbb N$. Define the function
\[ f_{m+1} (x) = \sum _{l=0} ^m \| \sigma_{k_{m+1}} (c) - \sigma_{k_m} (c) \| _{\mathcal E _c} \]
for every $m \in \mathbb N$ and notice that $\| f_m \| _{L^2 (\mathcal C_t)} \le 1$. As a consequence of the monotone convergence theorem, $(f_m) _{m \ge 1}$ has a limit $f \in L^2 (\mathcal C_t)$, finite almost everywhere.

If $m \ge l \ge 0$, then for almost all $c \in \mathcal C_t$ we have
\[ \| \sigma_{k_m} (c) - \sigma_{k_l} (c) \| _{\mathcal E _c} \le \| \sigma_{k_m} (c) - \sigma_{k_m - 1} (c) \| _{\mathcal E _c} + \dots + \| \sigma_{k_l + 1} (c) - \sigma_{k_l} (c) \| _{\mathcal E _c} \le f(c) - f_{k_l} (c) \to 0 \ , \]
therefore for almost all $c \in \mathcal C_t$ the sequence $(\sigma_{k_m} (c)) _{m \in \mathbb N} \subset \mathcal E _c$ is Cauchy. Since the space $\mathcal E _c$ is, by construction, a Hilbert space, hence complete, it follows that for almost all $c \in \mathcal C_t$ there exists a unique element $\sigma(c) \in \mathcal E _c$ such that $\sigma_{k_m} (c) \to \sigma(c)$. We have already noticed that $\mathcal C_t$ may be covered by a countable family of trivialization domains (open balls) of $\mathcal E$; on each of them, $\sigma_{k_m} \to \sigma$ almost everywhere, therefore the restriction of $\sigma$ to each such trivialization domain is measurable. Since this family of trivialization domains is countable, it follows that $\sigma$ is measurable. Furthermore, passing to the limit in the inequality
\[ \| \sigma_{k_m} (c) - \sigma_{k_l} (c) \| _{\mathcal E _c} \le f(c) - f_{k_l} (c) \le f(c) \]
we obtain $\| \sigma(c) - \sigma_{k_l} (c) \| _{\mathcal E _c} \le f(c)$, whence
\[ | \|\sigma\| (c) - \|\sigma_{k_l}\| (c) | \le \| \sigma(c) - \sigma_{k_l} (c) \| _{\mathcal E _c} \le f(c) \ , \]
hence $\| \sigma \| \in L^2(\mathcal C_t)$, which means that $\sigma \in \Gamma^2 (\mathcal E)$. Finally, applying the dominated convergence theorem to the sequence $c \mapsto \| \sigma(c) - \sigma_{k_l} (c) \| _{\mathcal E _c} ^2$ (which is dominated by $f^2$), we conclude that $\sigma_{k_m} \to \sigma$ in $\Gamma^2 (\mathcal E)$, so $\sigma_k \to \sigma$ in $\Gamma^2 (\mathcal E)$.

That the dual of $\Gamma^2 (\mathcal E)$ is $\Gamma^2 (\mathcal E^*)$ is now easy, using the same techniques.
\end{proof}

More generally, and along the same lines of thought, one may introduce the space $\Gamma^p (\mathcal E)$ for every $p \in [1, \infty]$, which will be a Banach space. In particular, $\Gamma^q (\mathcal E) \subseteq \Gamma^p (\mathcal E)$ if $p \le q$, because the Wiener measure is finite. Also, $\Gamma^p (\mathcal E ^*)$ is the dual of $\Gamma^{\frac p {p-1}} (\mathcal E)$ for every $p \in (1, \infty]$. The proofs are analogous to those for the spaces $L^p$, the latter being found, for instance, in chap.4 of \cite{Brezis11}.

\begin{theorem}
The space $\operatorname {Cyl} _t (\mathcal E)$ of continuous and bounded cylindrical sections is dense in $\Gamma^2 (\mathcal E)$.
\end{theorem}
\begin{proof}
The inclusion $\operatorname {Cyl} _t (\mathcal E) \subset \Gamma^2 (\mathcal E)$ is trivial: if $s : M^{2^k + 1} \to (\operatorname{End} E)^{\boxtimes (2^k + 1)}$ is essentially bounded, then
\[ \int _{\mathcal C_t} \| (s \circ \pi_k) (c) \| _{\mathcal E_c} \, \mathrm d w_t (c) \le \sup_{c \in \mathcal C_t} \| (s \circ \pi_k) (c) \| \, w_t (\mathcal C_t) < \infty \ . \]

Let now $\sigma' \in \operatorname {Cyl} _t (\mathcal E) ^\perp$; we shall show that $\sigma' = 0$. If $f \in \operatorname{Cyl} (\mathcal C_t)$ is a cylindrical function (the definition of which is given in \cite{Mustatea22}) and $\sigma \in \operatorname{Cyl}_t (\mathcal E)$ is a cylindrical section, then it is easy to show that $f \sigma \in \operatorname{Cyl}_t (\mathcal E)$ and, since $\sigma' \in \operatorname {Cyl} _t (\mathcal E) ^\perp$, we shall have in particular that
\[ 0 = \langle f \sigma, \sigma' \rangle _{\Gamma^2 (\mathcal E)} = \int _{\mathcal C_t} f(c) \, \langle \sigma(c), \sigma'(c) \rangle _{\mathcal E_c} \, \mathrm d w_t (c) \ . \]
Using theorem 2.1 in \cite{Mustatea22}, the cylindrical functions are dense in $L^2 (\mathcal C_t)$, so
\[ \int _{\mathcal C_t} f(c) \, \langle \sigma(c), \sigma'(c) \rangle _{\mathcal E_c} \, \mathrm d w_t (c) = 0 \]
for all $f \in L^2 (\mathcal C_t)$, whence we deduce that $\langle \sigma(c), \sigma'(c) \rangle _{\mathcal E_c} = 0$ for all $c$ in some co-null subset $C_\sigma \subseteq \mathcal C_t$.

Let $M = \bigcup _{i \in \mathbb N} V_i '$ be a cover of $M$ with open trivialization domains for $E$. Let $V_0 = V_0 '$ and $V_i = V_i ' \setminus (V_0 \cup \dots \cup V_{i-1})$ for $i \ge 1$; these subsets will be measurable, pairwise disjoint, trivialization domains. Let $\{ \eta _i ^1, \dots, \eta _i ^{r^2} \}$ be a measurable orthonormal frame in $\operatorname{End} E | _{V_i}$ in which $\eta _i ^1 (x) = \operatorname{Id} _{E_x}$ for all $x \in V_i$. Defining $\eta ^l$ by $\eta ^l | _{V_i} = \eta _i ^l$ for all $1 \le l \le r^2$ and $i \in \mathbb N$, we obtain a global measurable orthonormal frame $\{ \eta ^1, \dots, \eta ^{r^2} \}$ in $\operatorname{End} E$ made of sections from $\Gamma^\infty (\operatorname {End} E)$, in which $\eta ^1 (x) = \operatorname{Id} _{E_x}$ for all $x \in M$.

For every $k \in \mathbb N$ and $1 \le j_0, \dots, j_{2^k} \le r^2$ define
\[ \sigma _{j_0 \dots j_{2^k}} (c) = \eta ^{j_0} (c(0)) \otimes \eta ^{j_1} (c(\frac t {2^k})) \otimes \dots \otimes \eta ^{j_{2^k}} (c(t)) \]
and notice that $\sigma _{j_0 \dots j_{2^k}} \in \operatorname{Cyl}_t (\mathcal E)$ and that the subset $\{ \sigma _{j_0 \dots j_{2^k-1}} (c) \mid k \in \mathbb N, \ 1 \le j_0, \dots, j_{2^k} \le r^2 \}$ is a countable orthonormal basis in the fiber $\mathcal E_c$ for all $c \in \mathcal C_t$. We then deduce that there exists a co-null subset $C_{j_0 \dots j_{2^k}} \subseteq \mathcal C_t$ such that
\[ \langle \sigma _{j_0 \dots j_{2^k-1}} (c), \sigma' (c) \rangle _{\mathcal E_c} = 0 \]
for all $c \in C_{j_0 \dots j_{2^k}}$, all $k \in \mathbb N$ and $1 \le j_0, \dots, j_{2^k} \le r^2$. If
\[ C = \bigcap _{k \in \mathbb N} \bigcap _{1 \le j_1, \dots, j_{2^k} \le r^2} C_{j_0 \dots j_{2^k}} \]
then $C$ will be co-null and
\[ \langle \sigma _{j_0 \dots j_{2^k-1}} (c), \sigma' (c) \rangle _{\mathcal E_c} = 0 \]
for all $c \in C$, for all $k \in \mathbb N$ and $1 \le j_0, \dots, j_{2^k} \le r^2$, whence $\langle u, \sigma' (c) \rangle _{\mathcal E_c} = 0$ for all $u \in \mathcal E_c$, hence $\sigma' (c) = 0$ for all $c \in C$, so $\sigma'=0$ in $\Gamma^2 (\mathcal E)$, so $\operatorname {Cyl} _t (\mathcal E) ^\perp = 0$, meaning that $\operatorname {Cyl} _t (\mathcal E)$ is dense in $\Gamma^2 (\mathcal E)$.
\end{proof}

In what follows, the main technical result (theorem \ref{application of Chernoff's theorem}) will be based upon the use of Chernoff's approximation theorem for $1$-parameter semigroups. This, in turn, will require us to work on compact subsets of $M$ in order to be able to guarantee the boundedness of certain complicated continuous functions. For this reason we shall consider an exhaustion $M = \bigcup _{i \in \mathbb N} U_i$ of $M$ with relatively compact connected domains with smooth boundary, such that $x_0 \in U_0$. In particular, these domains will be Riemannian manifolds, therefore all the above considerations will apply to them, too. All the mathematical objects on $U_i$ obtained as restrictions of some extrinsic objects will be represented visually by the restriction symbol (such as in, for instance, the bundle $E | _{U_i}$), and all the objects intrinsically associated to $U_i$ will carry the index $(i)$ (for instance: the heat kernel associated to the connection $\nabla$ in $E | _{U_i}$ will be $h_\nabla ^{(i)}$, the Laplacean understood as the generator of the heat semigroup acting in $C (\overline {U_i})$ will be $L^{(i)}$ etc.).

For each $i \in \mathbb N$ we shall consider the space
\[ \mathcal C_t (\overline{U_i}) = \{ c \in \mathcal C_t \mid c([0,t]) \subseteq \overline {U_i} \} \]
endowed with the restriction of the distance $D$ introduced on $\mathcal C_t$. The natural measure on $\mathcal C_t (\overline{U_i})$ will \textit{not} be the restriction of the Wiener measure $w_t$, but rather the intrinsic Wiener measure $w_t ^{(i)}$ obtained from the intrinsic heat kernel $h ^{(i)}$ on $\overline {U_i}$. It is elementary that $\mathcal C_t (\overline{U_i})$ is closed (and therefore Borel) in $\mathcal C_t$: the evaluation map $\operatorname{ev} : [0,t] \times \mathcal C_t \to M$ defined by $\operatorname{ev} (s, \gamma) = \gamma(s)$ is obviously continuous, whence
\[ \mathcal C_t (\overline {U_i}) = \{ \gamma \in C_t \mid \gamma(s) \in \overline {U_i} \ \forall s \in [0,t] \} = \bigcap _{s \in [0,t]} \operatorname{ev} (s, \cdot) ^{-1} (\overline {U_i}) \]
is obviously closed. One shows similarly that, if $i \le j$, then $\mathcal C_t (\overline {U_i})$ is closed in $\mathcal C_t (\overline {U_j})$. It is also known that $w_t ^{(i)} \le w_t | _{\mathcal C_t (\overline{U_i})}$. For details about the Wiener measure, the article \cite{BP11} contains all the necessary constructions and explanations; note that the constructions therein are not probabilistic, but functional-analytic, therefore our project of a purely functional-analytic construction of the stochastic parallel transport is not compromised.

In the following, we shall define a continuous linear functional on $\Gamma^2 (\mathcal E | _{\mathcal C_t (\overline {U_i})})$ to which, by Riesz's representation theorem, there will correspond a section from $\Gamma^2 (\mathcal E^* | _{\mathcal C_t (\overline {U_i})})$ which will be seen to be intimately linked to the stochastic parallel transport. Fix $\omega \in E_{x_0} ^*$ and $\eta \in \Gamma_{cb} (E)$, and define the functional $W_{t, \omega, \eta} ^{(i)}$ on continuous and bounded cylindrical sections as follows: if $s : \overline {U_i} ^{2^k + 1} \to (\operatorname{End} E)^{\boxtimes (2^k + 1)} | _{U^{2^k + 1}}$ is a continuous and bounded section, define
\begin{align*}
W_{t, \omega, \eta} ^{(i)} (s \circ \pi_k) = & \int _{U_i} \mathrm d x_1 \dots \int _{U_i} \mathrm d x_{2^k} \left[ \omega \otimes h_\nabla ^{(i)} \left( \frac t {2^k}, x_0, x_1 \right) \otimes \dots \right. \\
& \left. \dots \otimes h_\nabla ^{(i)} \left( \frac t {2^k}, x_{2^k-1}, x_{2^k} \right) \otimes \eta(x_{2^k}) \right] \cdot s(x_0, x_1, \dots, s_{x_{2^k}}) \ .
\end{align*}
The dot inside the integral denotes not a scalar product but a tensor contraction which, in order to be understood, requires a brief discussion. The term $\omega \otimes h_\nabla \left( \frac t {2^k}, x_0, x_1 \right) \otimes \dots \otimes h_\nabla \left( \frac t {2^k}, x_{2^k-1}, x_{2^k} \right) \otimes \eta(x_{2^k})$ belongs to the space $E_{x_0} ^* \otimes (E_{x_0} \otimes E_{x_1} ^*) \otimes \dots \otimes (E_{x_{2^k - 1}} \otimes E_{x_{2^k}} ^*) \otimes E_{x_{2^k}}$ which is naturally isomorphic to $(E_{x_0} ^* \otimes E_{x_0}) \otimes \dots \otimes (E_{x_{2^k}} ^* \otimes E_{x_{2^k}})$, which in turn is isomorphic to $(\operatorname{End} E^*) _{x_0} \otimes \dots \otimes (\operatorname{End} E^*) _{x_{2^k}}$ (notice that the latter isomorphism is not the natural one, but the natural one multiplied by a normalization factor, because the scalar product of two endomorphisms has been defined such that the identity should have norm $1$). In turn, $s(x_0, \dots, x_{2^k})$ belongs to the space $(\operatorname{End} E) _{x_0} \otimes \dots \otimes (\operatorname{End} E) _{x_{2^k}}$, therefore the term on the left of the dot may be naturally applied to the one on the right of the dot, this being the meaning of the tensor contraction inside the integral.

Let us show that, indeed, the functional is well defined. First, if $l > k$ then there exists a projection $\pi_{kl} : M^{2^l+1} \to M^{2^k+1}$ given by $\pi_{kl} (x_0, \dots, x_{2^l}) = (x_{j 2^{l-k}})_{0 \le j \le 2^k}$, so that $s \circ \pi_k = (s \circ \pi_{kl}) \circ \pi_l$. This shows that a cylindrical section may have multiple writings of the form $s \circ \pi_k$. This fact is fortunately compensated inside the integral by the convolution property of the kernel $h_\nabla ^{(i)}$, which insures that the formula of definition of $W_{t, \omega, \eta} ^{(i)}$ does not depend on the writing of the cylindrical sections.

In order to show that the integral in the definition of $W_{t, \omega, \eta} ^{(i)}$ exists, let us notice that
\begin{gather*}
\left| \left[ \omega \otimes h_\nabla ^{(i)} \left( \frac t {2^k}, x_0, x_1 \right) \otimes \dots \otimes h_\nabla ^{(i)} \left( \frac t {2^k}, x_{2^k-1}, x_{2^k} \right) \otimes \eta(x_{2^k}) \right] \cdot s(x_0, x_1, \dots, s_{x_{2^k}}) \right| \le \\
\left\| \omega \otimes h_\nabla ^{(i)} \left( \frac t {2^k}, x_0, x_1 \right) \otimes \dots \otimes h_\nabla ^{(i)} \left( \frac t {2^k}, x_{2^k-1}, x_{2^k} \right) \otimes \eta(x_{2^k}) \right\| \ \| s(x_0, x_1, \dots, s_{x_{2^k}}) \| \ ,
\end{gather*}
each of the two norms being considered in the appropriate space. We shall leave the second as it is, but we shall work on the first one. Let us consider the orthonormal basis $\{ e^i _1, \dots, e^i _r \}$ in each fiber $E_{x_i}$, and the dual basis $\{ f_i ^1, \dots, f_i ^r \}$ in each fiber $E_{x_i} ^*$. In these bases we have (using Einstein's summation convention) $\omega = \omega_{i'_0} \, f_0 ^{i'_0}$, $h_\nabla ^{(i)} \left( \frac t {2^k}, x_{j-1}, x_j \right) = h^{(i), i_{j-1}} _{i'_j} \, e^{j-1} _{i_{j-1}} \otimes f_j ^{i'_j}$ (for $1 \le j \le 2^k$) and $\eta(x_{2^k}) = \eta^{i_{2^k}} \, e^{2^k} _{i_{2^k}}$, hence
\begin{gather*}
\left\| \omega \otimes h_\nabla ^{(i)} \left( \frac t {2^k}, x_0, x_1 \right) \otimes \dots \otimes h_\nabla ^{(i)} \left( \frac t {2^k}, x_{2^k-1}, x_{2^k} \right) \otimes \eta(x_{2^k}) \right\| ^2 _{(\operatorname {End} E^*) _{x_0} \otimes \dots \otimes (\operatorname {End} E^*) _{x_{2^k}}} = \\
= \omega_{i'_0} \, h^{(i), i_0} _{i'_1} \dots h^{(i), i_{2^k-1}} _{i'_{2^k}} \, \eta^{i_{2^k}} \, \overline {\omega_{j'_0} \, h^{(i), j_0} _{j'_1} \dots h^{(i), j_{2^k-1}} _{j'_{2^k}} \, \eta^{j_{2^k}}} \cdot \\
\cdot \langle f_0 ^{i'_0} \otimes e^0 _{i_0} \otimes \dots \otimes f_{2^k} ^{i'_{2^k}} \otimes e^{2^k} _{i_{2^k}}, f_0 ^{j'_0} \otimes e^0 _{j_0} \otimes \dots \otimes f_{2^k} ^{j'_{2^k}} \otimes e^{2^k} _{j_{2^k}} \rangle _{(\operatorname {End} E^*) _{x_0} \otimes \dots \otimes (\operatorname {End} E^*) _{x_{2^k}}} = \\
= \omega_{i'_0} \, h^{(i), i_0} _{i'_1} \dots h^{(i), i_{2^k-1}} _{i'_{2^k}} \, \eta^{i_{2^k}} \, \overline {\omega_{j'_0} \, h^{(i), j_0} _{j'_1} \dots h^{(i), j_{2^k-1}} _{j'_{2^k}} \, \eta^{j_{2^k}}} \cdot \\
\cdot \langle f_0 ^{i'_0} \otimes e^0 _{i_0}, f_0 ^{j'_0} \otimes e^0 _{j_0} \rangle _{(\operatorname{End} E^*) _{x_0}} \dots \langle f_{2^k} ^{i'_{2^k}} \otimes e^{2^k} _{i_{2^k}}, f_{2^k} ^{j'_{2^k}} \otimes e^{2^k} _{j_{2^k}} \rangle _{(\operatorname{End} E^*) _{x_{2^k}}} = \\
= \left( \frac 1 r \right) ^{2^k + 1} \omega_{i'_0} \, h^{(i), i_0} _{i'_1} \dots h^{(i), i_{2^k-1}} _{i'_{2^k}} \, \eta^{i_{2^k}} \, \overline {\omega_{j'_0} \, h^{(i), j_0} _{j'_1} \dots h^{(i), j_{2^k-1}} _{j'_{2^k}} \, \eta^{j_{2^k}}} \cdot \\
\cdot \langle f_0 ^{i'_0} \otimes e^0 _{i_0}, f_0 ^{j'_0} \otimes e^0 _{j_0} \rangle _{E_{x_0} ^* \otimes E_{x_0}} \dots \langle f_{2^k} ^{i'_{2^k}} \otimes e^{2^k} _{i_{2^k}}, f_{2^k} ^{j'_{2^k}} \otimes e^{2^k} _{j_{2^k}} \rangle _{E_{x_{2^k}} ^* \otimes E_{x_{2^k}}} = \\
= \left( \frac 1 r \right) ^{2^k + 1} \sum_{i'_0} |\omega_{i'_0}|^2 \sum_{i_0, i'_1} |h^{(i), i_0} _{i'_1}|^2 \dots \sum_{i_{2^k-1}, i'_{2^k}} |h^{i_{2^k-1}} _{i'_{2^k}}|^2 \sum_{i_{2^k}} |\eta^{i_{2^k}}|^2 = \\
= \left( \frac 1 r \right) ^{2^k + 1} \| \omega \| ^2 _{E_{x_0} ^*} \left\| h_\nabla ^{(i)} \left( \frac t {2^k}, x_0, x_1 \right) \right\| ^2 _{E_{x_0} \otimes E_{x_1}^*} \dots \\
\dots \left\| h_\nabla ^{(i)} \left( \frac t {2^k}, x_{2^k - 1}, x_{2^k} \right) \right\| ^2 _{E_{x_{2^k-1}} \otimes E_{x_{2^k}}^*} \| \eta (x_{2^k}) \| ^2 _{E_{x_{2^k}}} \le \\
\le \frac 1 r \| \omega \| ^2 _{E_{x_0} ^*} \left\| h_\nabla ^{(i)} \left( \frac t {2^k}, x_0, x_1 \right) \right\| ^2 _{op} \dots \left\| h_\nabla ^{(i)} \left( \frac t {2^k}, x_{2^k - 1}, x_{2^k} \right) \right\| ^2 _{op} \| \eta (x_{2^k}) \| ^2 _{E_{x_{2^k}}} \le \\
\le \frac 1 r \| \omega \| ^2 _{E_{x_0} ^*} h ^{(i)} \left( \frac t {2^k}, x_0, x_1 \right) ^2 \dots h ^{(i)} \left( \frac t {2^k}, x_{2^k - 1}, x_{2^k} \right) ^2 \| \eta (x_{2^k}) \| ^2 _{E_{x_{2^k}}} \ ,
\end{gather*}
where we have used that $\| A \| _{V \otimes U^*} \le \sqrt r \| A \| _{op}$ for any linear map $A : U \to V$ between vector spaces of dimension $r$. We have also used the diamagnetic inequality $\| h_\nabla ^{(i)} (s, x, y) \| _{op} \le h^{(i)} (s,x,y)$.

We conclude that
\begin{gather*}
\left| \left[ \omega \otimes h_\nabla ^{(i)} \left( \frac t {2^k}, x_0, x_1 \right) \otimes \dots \otimes h_\nabla ^{(i)} \left( \frac t {2^k}, x_{2^k-1}, x_{2^k} \right) \otimes \eta(x_{2^k}) \right] \cdot s(x_0, x_1, \dots, s_{x_{2^k}}) \right| \le \\
\le \frac 1 {\sqrt r} \| \omega \| _{E_{x_0} ^*} h ^{(i)} \left( \frac t {2^k}, x_0, x_1 \right) \dots h ^{(i)} \left( \frac t {2^k}, x_{2^k - 1}, x_{2^k} \right) \| \eta (x_{2^k}) \| _{E_{x_{2^k}}} \| s(x_0, x_1, \dots, s_{x_{2^k}}) \| \ ,
\end{gather*}
hence that $W_{t, \omega, \eta} ^{(i)}$ is well defined, trivially linear, and that
\begin{align*}
|W_{t, \omega, \eta} ^{(i)} (s \circ \pi_k)| & \le \frac 1 {\sqrt r} \| \omega \| _{E_{x_0} ^*} \int _{\mathcal C_t (\overline {U_i})} \| (\eta (c(t)) \| (s \circ \pi_k) (c) \| \, \mathrm d w_t ^{(i)} (c) \le \\
& \le \frac 1 {\sqrt r} \| \omega \| _{E_{x_0} ^*} [(\mathrm e ^{-t L^{(i)}} \| \eta \| ^2) (x_0)] ^{\frac 1 2} \, \| (s \circ \pi_k) (c) \| _{\Gamma^2 (\mathcal E | _{\mathcal C_t (\overline {U_i})})} \ ,
\end{align*}
where $\| \eta \|$ denotes the function $M \ni x \mapsto \| \eta(x) \| _{E_x} \in [0, \infty)$.

Since the continuous and bounded cylindrical sections are dense in $\Gamma^2 (\mathcal E | _{\mathcal C_t (\overline {U_i})})$, it follows that $W_{t, \omega, \eta} ^{(i)}$ extends uniquely to a continuous linear functional on this space, therefore there exists a unique $\rho_{t, \omega, \eta} ^{(i)} \in \Gamma^2 (\mathcal E ^* | _{\mathcal C_t (\overline {U_i})})$ such that
\[ W_{t, \omega, \eta} ^{(i)} (\sigma) = \int _{\mathcal C_t (\overline {U_i})} \rho_{t, \omega, \eta} ^{(i)} (c) (\sigma (c)) \, \mathrm d w_t ^{(i)} (c) \]
for every $\sigma \in \Gamma^2 (\mathcal E | _{\mathcal C_t (\overline {U_i})})$. Furthermore, $\| \rho_{t, \omega, \eta} ^{(i)} \| _{\Gamma^2 (\mathcal E ^* | _{\mathcal C_t (\overline {U_i})})} \le \frac 1 {\sqrt r} \| \omega \| _{E_{x_0} ^*} [(\mathrm e ^{-t L^{(i)}} \| \eta \| ^2) (x_0)] ^{\frac 1 2}$.

In the following we shall try to uncover some of the geometrical properties of $\rho_{t, \omega, \eta} ^{(i)}$; more precisely, we shall investigate its connection with the parallel transport in $E$. To this end, let us define the "cut-off parallel transport" $P(x,y) : E_y \to E_x$ for every $(x,y) \in M \times M$ by:
\begin{itemize}[wide]
\item $P(x,y) = $ the parallel transport in $E$ from $y$ to $x$, whenever there exists a unique minimizing geodesic in $M$ defined on $[0,1]$ between $x$ and $y$,
\item $P(x,y) = 0$ otherwise.
\end{itemize}
Let us notice that $P$ so defined is a section in the external tensor product bundle $E \boxtimes E^* \to M \times M$. Since the subset
\[ \{ (x,y) \in M \times M \mid \text{there exists a unique minimizing geodesic between } x \text{ and } y \text{ defined on } [0,1] \} \]
is open in $M \times M$, it will be Borel measurable. Since the section $(x,y) \mapsto P(x,y)$ in $E \boxtimes E^*$ is continuous on this subset, $P$ will be a measurable section in this bundle.

With $P$ so defined, define
\[ P_{t, \omega, \eta, k} (c) = \omega \otimes P \left( c(0), c \left( \frac t {2^k} \right) \right) \otimes \dots \otimes P \left( c \left( \frac {(2^k-1) t} {2^k} \right), c(t) \right) \otimes \eta (c(t)) \]
for every curve $c \in \mathcal C_t$ and every $k \in \mathbb N$. Since $P$ is measurable, with operator norm bounded by $1$ at every point of $M \times M$, we conclude that $P_{t, \omega, \eta, k}$ is a measurable and bounded cylindrical section in the bundle $\mathcal E ^*$. We shall show that $\rho_{t, \omega, \eta}$ is the limit of the sequence $(P_{t, \omega, \eta, k}) _{k \in \mathbb N}$ in the norm topology of $\Gamma^2 (\mathcal E ^*)$.

We shall need to use $P$ (which is not smooth) in contexts requiring differential calculus methods; in order to do this, we shall now introduce some smooth cut-off functions. Let $\kappa : [0, \infty) \to [0,1]$ be a smooth function such that $\kappa | _{[0, \frac 1 3]} = 1$ and $\kappa | _{[\frac 1 2, \infty)} = 0$. Let $\operatorname{injrad}_U : U \to (0, \infty)$ be the injectivity radius function on $U$; we emphasize that this is not the restriction of $\operatorname{injrad}_M$ to $U$, but rather is is computed intrinsically, using the restriction of the Riemannian structure on $U$ (for basic details about the injectivity radius, see p.118 of \cite{Chavel06}). Being continuous and strictly positive, we may find a smooth function $\operatorname{rad} : U \to (0, \infty)$ such that $\operatorname{rad}(x) < \operatorname{injrad}_U (x)$. In particular, $\operatorname{rad}(x) \le d_U (x, \partial U)$ (the distance up to the boundary of $U$, computed using the intrinsic distance $d_U$ of $U$, not using the distance $d$ restricted to $U$). We may now finally define the desired cut-off function $\chi : U \times U \to [0,1]$ by $\chi (x,y) = \kappa \left( \frac {d_U (x,y)^2} {\operatorname{rad}(x) ^2} \right)$. Notice that $\chi$ is smooth (the square is necessary in order to guarantee the smoothness close to the points with $y=x$). We shall also define the cut-off function $\chi_k : \mathcal C_t (\overline U) \to [0,1]$ by
\[ \chi_k (c) = \chi \left( c(0), c\left( \frac t {2^k} \right) \right) \chi \left( c \left( \frac t {2^k} \right), c \left( \frac {2t} {2^k} \right) \right) \dots \chi \left( c \left( \frac {(2^k - 1)t} {2^k} \right), c(t) \right) \]
for every $k \ge 0$.

If $h^{(i)}$ is the intrinsic heat kernel of $\overline {U_i}$, the operators defined by
\[ C(\overline {U_i}) \ni f \mapsto \int _{U_i} h^{(i)} (t, \cdot ,y) \, f(y) \, \mathrm d y \in C(\overline {U_i}) \]
together with the identity operator form a strongly continuous one-parameter semigroup in $C(\overline {U_i})$. This will have a generator (closed operator) that we shall denote by $L^{(i)}$, densely defined, with the domain given by (see \cite{Davies80}, chap. 1)
\[ \operatorname{Dom} (L^{(i)}) = \left\{ f \in C(\overline {U_i}) ; \lim _{t \to 0} \frac 1 t \left( \int _U h^{(i)} (t, \cdot ,y) \, f(y) \, \mathrm d y - f \right) \in C(\overline {U_i}) \right\} \ . \]
We shall denote this semigroup by $(\mathrm e ^{-s L^{(i)}}) _{s \ge 0}$. Integrating twice by parts, it is obvious that $C_0 ^\infty (U_i) \subset \operatorname{Dom} (L^{(i)})$. An essential domain for $L^{(i)}$ is
\[ \mathcal E = \bigcup _{s > 0} \mathrm e ^{-s L^{(i)}} (C(\overline {U_i})) \ . \]
Since the heat semigroup is smoothing (again, one may use \cite{Mizohata57}, or one's favourite Sobolev spaces techniques, to see this), the functions in $\mathcal E$ will be smooth. Since $h^{(i)}$ vanishes on the boundary $\partial {U_i}$, the functions in $\mathcal E$ will also vanish on $\partial {U_i}$.

With exactly the same arguments, but using now the integral kernel $h_\nabla ^{(i)}$ instead of $h ^{(i)}$, we shall obtain a semigroup acting on $\Gamma_c (\overline {U_i})$, the generator of which will be denoted by $L_\nabla ^{(i)}$, and the domain of which will contain $\Gamma_0 ^\infty (E | _{U_i})$. This semigroup will be denoted $(\mathrm e ^{-s L^{(i)} _\nabla}) _{s \ge 0}$.

The crucial tool to be used in the following will be Chernoff's theorem (lemma 3.28 in \cite{Davies80}). For the reader's convenience, we shall give its statement here.

\begin{theorem}[Chernoff]
Assume that $(R_t) _{t \ge 0}$ is a family of contractions in a Banach space $X$, with $R_0 = \operatorname{Id}_X$. Let $\mathcal E \subseteq X$ be an essential domain for the generator $L$ of a strongly continuous one-parameter semigroup $(\mathrm e ^{-t L}) _{t \ge 0}$ on $X$. If $\lim _{t \to 0} \frac 1 t (R_t f - f) = -L f$ for every $f \in \mathcal E$, then $\mathrm e ^{-t L} = \lim _{k \to \infty} \big( R_{\frac t k} \big) ^k$ strongly for every $t \ge 0$. Furthermore, the convergence is uniform with respect to $t$ on bounded subsets of $[0, \infty)$.
\end{theorem}

With all these preparations, we are ready now for the main technical result of this work, from which all the developments announced in the introduction will unravel.

\begin{theorem} \label{application of Chernoff's theorem}
If $R_t ^{(i)} : C(\overline {U_i}) \to C(\overline {U_i})$, with $t \ge 0$, is the family of operators given by $R_0 ^{(i)} f = f$ and
\[ (R_t ^{(i)} f) (x) = \frac 1 r \int _{U_i} \langle h_\nabla ^{(i)} (t, x, y), \chi(x,y) P(x,y) \rangle _{E_x \otimes E_y ^*} f(y) \, \mathrm d y \]
for $f \in C(\overline {U_i})$, then $\lim _{k \to \infty} \left( R_ {\frac t k} ^{(i)} \right) ^k f = \mathrm e^{-t L^{(i)}} f$ for every $f \in C(\overline {U_i})$, uniformly with respect to $t$ from compact subsets of $[0, \infty)$.
\end{theorem}
\begin{proof}
The proof reduces to the verification of the hypotheses in Chernoff's theorem. To begin with, notice that for $t>0$
\begin{align*}
|(R_t ^{(i)} f) (x)| & \le \int _{U_i} \chi(x,y) \, \frac 1 {\sqrt r} \| h_\nabla ^{(i)} (t,x,y) \| _{E_x \otimes E_y ^*} \, \frac 1 {\sqrt r} \| P(x,y) \| _{E_x \otimes E_y ^*} \, |f(y)| \, \mathrm d y \le \\
& \le \int _{U_i} \chi(x,y) \, \| h_\nabla ^{(i)} (t,x,y) \| _{op} \, \| P(x,y) \| _{op} \, |f(y)| \, \mathrm d y \le \\
& \le \int _{U_i} h ^{(i)} (t,x,y) |f(y)| \, \mathrm d y \ , \le \| f \| _{C(\overline{U_i})} \ ,
\end{align*}
so $R_t ^{(i)}$ is a contraction for every $t \ge 0$ (we have used again the fact that $\| A \| _{V \otimes U^*} \le \sqrt r \| A \| _{op}$, the diamagnetic inequality and the obvious inequalities $\chi \le 1$ și $\| P(x,y) \| _{op} \le 1$). It remains to show that $\lim _{t \to 0} \| \frac 1 t (R_t ^{(i)} f - f) + L ^{(i)} f \| _{C(\overline{U_i})} = 0$ for every $f \in \mathcal E$; to this end, let us show first that $(R_t ^{(i)} f) (x)$ is smooth with respect to $t$ for every $x \in U$. If $\operatorname{Trace}$ denotes the trace in $\operatorname{End} E_x$, notice that
\begin{align*}
(R_t ^{(i)} f) (x) & = \frac 1 r \int _{U_i} \langle h_\nabla ^{(i)} (t, x, y), \chi(x,y) P(x,y) \rangle _{E_x \otimes E_y ^*} f(y) \, \mathrm d y = \\
& = \frac 1 r \int _{U_i} \operatorname{Trace} [h_\nabla ^{(i)} (t, x, y) \chi(x,y) P(x,y) ^* ] f(y) \, \mathrm d y = \\
& = \frac 1 r \int _{U_i} \operatorname{Trace} [h_\nabla ^{(i)} (t, x, y) \chi(x,y) P(y,x) ] f(y) \, \mathrm d y = \\
& = \frac 1 r \operatorname{Trace} \{ \mathrm e ^{-t L^{(i)} _\nabla} [\chi(x, \cdot) \, P(\cdot, x) \, f] \} (x) \ .
\end{align*}
Examining the construction of $\chi$, it is clear that $\chi(x, \cdot) \, P(\cdot, x)$ is a smooth section in $E | _{\overline {U_i}}$ with compact support, the possible singularities of $P(\cdot, x)$ being away from the support of $\chi(x, \cdot)$; since $f$ is smooth, being from $\mathcal E$, their product is a smooth section in $E | _{\overline {U_i}}$ with compact support, therefore in the domain of every power of $L^{(i)} _\nabla$. Under these conditions, we know from the general theory of $1$-parameter $C_0$-semigroups in Banach spaces that the map
\[ [0, \infty) \ni t \mapsto \mathrm e ^{-t L^{(i)} _\nabla} [\chi(x, \cdot) \, P(x, \cdot) \, f] \in \Gamma_c (E | _{\overline{U_i}}) \]
is smooth. If $\{e_1, \dots, e_r\}$ is an orthonormal basis in $E_x$, and if $\delta_x$ is the Dirac measure concentrated at $x$, then $\delta_x \otimes e_i$ is easily seen to be a continuous linear functional on $\Gamma_c (E | _{\overline{U_i}})$ for each $1 \le i \le r$; since
\[ \{ \mathrm e ^{-t L^{(i)} _\nabla} [\chi(x, \cdot) \, P(x, \cdot) \, f] \} (x) = \sum _{i=1} ^r (\delta_x \otimes e_i) \left( \mathrm e ^{-t L^{(i)} _\nabla} [\chi(x, \cdot) \, P(x, \cdot) \, f] \right) e_i \ , \]
the smoothness of the map $[0, \infty) \ni t \mapsto \{ \mathrm e ^{-t L^{(i)} _\nabla} [\chi(x, \cdot) \, P(x, \cdot) \, f] \} (x) \in E_x$ is clear, whence the smoothness of the function $[0, \infty) \ni t \mapsto (R_t ^{(i)} f) (x) \in \mathbb C$ follows immediately.

Expanding with respect to $t$ we have, for every $x \in \overline{U_i}$,
\begin{equation}
(R_t ^{(i)} f) (x) = f(x) + \partial_t |_{t=0} (R_t ^{(i)} f) (x) \, t + \int _0 ^t (t-s) \partial _s ^2 (R_s ^{(i)} f) (x) \, \mathrm d s \ . \label{Taylor expansion}
\end{equation}

For the calculation of the first derivative of $(R_t ^{(i)} f) (x)$ at $t=0$ we have
\begin{align*}
\partial_t |_{t=0} (R_t ^{(i)} f) (x) & = \partial_t |_{t=0} \int _{U_i} \frac 1 r \, \operatorname{Trace} [h_\nabla ^{(i)} (t,x,y) \, \chi(x,y) \, P(x,y)^* ] \, f(y) \, \mathrm d y = \\
& = \frac 1 r \lim _{t \to 0} \int _{U_i} \operatorname{Trace} \{[-L_{\nabla, (y)} ^{(i)} h_\nabla ^{(i)} (t,x,y)] \, \chi(x,y) \, P(y,x) \, f(y) \} \, \mathrm d y = \\
& = \frac 1 r \operatorname{Trace} \left\{ \lim _{t \to 0} \int _{U_i} h_\nabla ^{(i)} (t,x,y) \{ -L_{\nabla, (y)} ^{(i)} [\chi(x,y) \, P(y,x) \, f(y)] \} \, \mathrm d y \right\} = \\
& = \frac 1 r \operatorname{Trace} \left\{ \{ -L_{\nabla, (y)} ^{(i)} [\chi(x,y) \, P(y,x) \, f(y)] \} _{y=x} \right\} .
\end{align*}
Some clarifications about the above calculations are in order. First, the notation $L_{\nabla, y} ^{(i)}$ means that the Laplacian acts with respect to $y \in \overline{U_i}$. Second, we have been able to move the Laplacian from acting on $h_{\nabla} ^{(i)} (t, x, \cdot)$ over to acting on the product $\chi(x, \cdot) \, P(\cdot, x) \, f$ because $\chi (x, \cdot)$ is smooth with compact support, and the other two factors are also smooth inside this support; this is, in fact, the only reason for which the introduction of $\chi$ in our reasoning was necessary.

Since $L_\nabla ^{(i)}$ is a local operator, since $\chi(x, \cdot) = 1$ near $x$, and since $\chi(x, \cdot) \, P(\cdot, x) \, f$ is smooth near $x$, we may replace $L_\nabla ^{(i)}$ with $\nabla^* \nabla$ and we may also drop $\chi$ in order to obtain the simpler formula
\[ \partial_t |_{t=0} (R_t ^{(i)} f) (x) = \frac 1 r \operatorname{Trace} \{ \nabla^* \nabla \, [P(\cdot, x) \, f] \} (x) . \]

Choosing an orthonormal basis $\{e_1, \dots, e_r\}$ in $E_x$, the above formula becomes
\[ \partial_t |_{t=0} (R_t ^{(i)} f) (x) = \frac 1 r \sum _{k=1} ^r \langle \nabla^* \nabla \, [P(\cdot, x) e_k \, f] (x), e_k \rangle _{E_x} . \]

Since $P$ is the parallel transport with respect to $\nabla$, its covariant derivative will be $0$, so
\[ \partial_t |_{t=0} (R_t ^{(i)} f) (x) = \frac 1 r \sum _{k=1} ^r \langle \nabla^* \, [P(\cdot, x) e_k \otimes \mathrm d f] (x), e_k \rangle _{E_x} . \]

It is a known result in Riemannian geometry that $\nabla ^* (\eta \otimes \alpha) = - \nabla _{\alpha^\sharp} \eta - (\mathrm d ^* \alpha) \eta$ for every real smooth $1$-form $\alpha$ and every smooth section $\eta$ in $E$, where $\alpha^\sharp$ is the tangent field dual to the $1$-form $\alpha$ under the usual "musical" isomorphisms. In particular, if $f$ is real
\[ \nabla ^* (\eta \otimes \mathrm d f) = - \nabla _{\operatorname{grad} f} \eta + (\Delta f) \eta \ , \]
whence it follows that
\begin{align*}
\partial_t |_{t=0} (R_t ^{(i)} f) (x) & = \frac 1 r \sum _{k=1} ^r \langle [-\nabla_{\operatorname{grad} f} P(\cdot, x) e_k] (x) + (\Delta f) (x) \, [P(\cdot, x) e_k] (x), e_k \rangle _{E_x} = \\
& = \frac 1 r \sum _{k=1} ^r \langle (\Delta f) (x) \, [P(x, x) e_k] (x), e_k \rangle _{E_x} = (\Delta f) (x) = -(L^{(i)} f) (x) \ ,
\end{align*}
where we have used again that $\nabla_{\operatorname{grad} f} P(\cdot, x) e_k = 0$ for the very same geometrical reasons as above. The result, obtained for real $f$, extends now trivially to complex $f$.

Returning to formula (\ref{Taylor expansion}), we have
\begin{align*}
\| (R_t ^{(i)} f - f) & + t \, L^{(i)} f \| _{C (\overline{U_i})} \le \sup _{x \in \overline{U_i}} \left| \int _0 ^t (t-s) \partial _s ^2 (R_s ^{(i)} f) (x) \, \mathrm d s \right| \le \frac {t^2} 2 \sup _{x \in \overline{U_i}} \sup_{s \in [0,t]} |\partial_s ^2 (R_s ^{(i)} f) (x)| \le \\
& \le \frac {t^2} 2 \sup _{x \in \overline{U_i}} \sup_{s \in [0,t]} \left| \frac 1 r \int _{U_i} \langle \partial_s ^2 h_\nabla ^{(i)} (s, x, y), \chi(x,y) P(x,y) \rangle _{E_x \otimes E_y ^*} f(y) \, \mathrm d y \right| \le \\
& \le \frac 1 r \frac {t^2} 2 \sup _{x \in \overline{U_i}} \sup_{s \in [0,t]} \left| \int _{U_i} \langle h_\nabla ^{(i)} (s, x, y), (L ^{(i)})^2 [\chi(x,y) P(x,y) \overline{f(y)}] \rangle _{E_x \otimes E_y ^*} \, \mathrm d y \right| \le \\
& \le \frac 1 r \frac {t^2} 2 \sup _{x \in \overline{U_i}} \sup_{y \in \overline{U_i}} |(L ^{(i)})^2 [\chi(x,y) P(x,y) \overline{f(y)}]| \ , 
\end{align*}
where in the last inequality we have used the diamagnetic inequality and the sub-markovianity of $h^{(i)}$. Since the function
\[ \overline{U_i} \times \overline{U_i} \ni (x,y) \mapsto |(L ^{(i)})^2 [\chi(x,y) P(x,y) \overline{f(y)}]| \in [0, \infty) \]
is smooth, the double supremum obtained in the last inequality will have a finite value $C \in [0, \infty)$, so
\[ \left\| \frac 1 t (R_t ^{(i)} f - f) + \, L^{(i)} f \right\| _{C (\overline{U_i})} \le \frac C {2r} \, t \to 0 \ , \]
which checks the last hypothesis in Chernoff's theorem, which we may now apply in order to obtain the conclusion of our theorem.
\end{proof}

The following corollary is essentially the above theorem in the trivial bundle $U_i \times \mathbb C$ endowed with the trivial connection given by differentiation (a situation in which the cut-off parallel transport $P$ may be replaced by the constant function $1$). The proof is essentially the same, but in an even simpler context, so we shall omit it.

\begin{corollary} \label{second application of Chernoff's theorem}
If $S_t ^{(i)} : C(\overline {U_i}) \to C(\overline {U_i})$, with $t \ge 0$, is the family of operators given by $S_0 ^{(i)} f = f$ and
\[ (S_t ^{(i)} f) (x) = \int _{U_i} h^{(i)} (t, x, y) \chi(x,y) f(y) \, \mathrm d y \]
for $f \in C(\overline {U_i})$, then $\lim _{k \to \infty} \left( S_ {\frac t k} ^{(i)} \right) ^k f = \mathrm e^{-t L^{(i)}} f$ for every $f \in C(\overline {U_i})$, uniformly with respect to $t$ from compact subsets of $[0, \infty)$.
\end{corollary}

\begin{remark}
Before going any further, let us pause for a moment and examine where in the above proof we have used the compactness of $\overline {U_i}$, and whether this compactness assumption is essential or not. It turns out that the only step in the proof where this assumption was used was in the bounding of the function
\[ \overline{U_i} \times \overline{U_i} \ni (x,y) \mapsto |(L ^{(i)})^2 [\chi(x,y) P(x,y) \overline{f(y)}]| \in [0, \infty) \ . \]
If instead of working on $\overline {U_i}$ we had worked on $M$, we would have needed to choose $f$ with the properties that:
\begin{itemize}[wide]
\item $f$ should be in the domain of $L^2$, where $L$ is the Friedrichs extension of the Laplace-Beltrami operator $-\Delta$ of $M$;
\item the product of $f$ with any compactly-supported smooth function should again be in the domain of $L^2$;
\item $f$ should have compact essential support, in order to guarantee the desired boundedness.
\end{itemize}
If $M$ had been metrically complete, then an essential domain for $L$ made of such functions would have been the space of compactly-supported smooth functions (see theorem 11.5 in \cite{Grigor'yan09}). For arbitrary Riemannian manifolds, though, no essential domain satisfying the above three conditions is known to the author, hence the need to treat the problem on relatively compact domains - a technical restriction that we shall see later on how to get rid of.
\end{remark}

The above results allow us to finally approach the statement that we were after, namely to prove that the sequence $(P_{t, \omega, \eta, k}) _{k \in \mathbb N}$ approximates $\rho_{t, \omega, \eta} ^{(i)}$.

\begin{theorem} \label{approximation on regular domains}
The sequence $(P_{t, \omega, \eta, k} | _{\mathcal C_t (\overline {U_i})}) _{k \in \mathbb N}$ converges to $\rho_{t, \omega, \eta} ^{(i)}$ in $\Gamma^2 (\mathcal E ^* | _{\mathcal C_t (\overline {U_i})})$, uniformly with respect to $t$ from bounded subsets of $(0, \infty)$.
\end{theorem}
\begin{proof}
In order to simplify the notations, we shall no longer indicate visually the restriction of functions or sections to $\mathcal C_t (\overline {U_i})$ where this is obvious. In the equality
\begin{align*}
\| \rho_{t, \omega, \eta} ^{(i)} - P_{t, \omega, \eta, k} \| _{\Gamma^2 (\mathcal E ^* | _{\mathcal C_t (\overline {U_i})})} ^2 & = \| \rho_{t, \omega, \eta} ^{(i)} \| _{\Gamma^2 (\mathcal E ^* | _{\mathcal C_t (\overline {U_i})})} ^2 - \langle \rho_{t, \omega, \eta} ^{(i)}, P_{t, \omega, \eta, k} \rangle _{\Gamma^2 (\mathcal E ^* | _{\mathcal C_t (\overline {U_i})})} - \\
& - \langle P_{t, \omega, \eta, k}, \rho_{t, \omega, \eta} ^{(i)} \rangle _{\Gamma^2 (\mathcal E ^* | _{\mathcal C_t (\overline {U_i})})} + \| P_{t, \omega, \eta, k} \| _{\Gamma^2 (\mathcal E ^* | _{\mathcal C_t (\overline {U_i})})} ^2
\end{align*}
the first term is less or equal than $\frac 1 r \| \omega \| _{E_{x_0} ^*} ^2 (\mathrm e ^{-t L^{(i)}} \| \eta \| ^2) (x_0)$. Performing majorizations similar to the ones made when we showed that $W_{t, \omega, \eta} ^{(i)}$ is well defined, in which we only replace $h_\nabla ^{(i)} \left( \frac t {2^k}, x_{j-1}, x_j \right)$ with $P(x_{j-1}, x_j)$ (the operator norm of which is less or equal than $1$), we obtain that
\[ \| P_{t, \omega, \eta, k} (c) \| _{\mathcal E ^* _c} ^2 \le \frac 1 r \| \omega \| _{E_{x_0}^*} ^2 \| \eta (c(t)) \| _{E_{c(t)}} ^2 \ , \]
whence we may bound the last term in the above right-hand side by
\begin{align*}
\| P & _{t, \omega, \eta, k} \| _{\Gamma^2 (\mathcal E ^* | _{\mathcal C_t (\overline {U_i})})} ^2 = \int _{\mathcal C_t (\overline {U_i})} \left\| \omega \otimes P \left( c(0), c \left( \frac t {2^k} \right) \right) \otimes \dots \right. \\
& \left. \dots \otimes P \left( c \left( \frac {(2^k-1) t} {2^k} \right), c(t) \right) \otimes \eta (c(t)) \right\| _{\mathcal E ^* _c} ^2 \, \mathrm d w_t ^{(i)} (c) \le \frac 1 r \| \omega \| _{E_{x_0} ^*} ^2 (\mathrm e ^{-t L^{(i)}} \| \eta \| ^2) (x_0) \ .
\end{align*}

We shall now show that $\lim _{k \to \infty} \langle \rho_{t, \omega, \eta} ^{(i)}, P_{t, \omega, \eta, k} \rangle _{\Gamma^2 (\mathcal E ^* | _{\mathcal C_t (\overline {U_i})})} = \frac 1 r \| \omega \| _{E_{x_0} ^*} ^2 (\mathrm e ^{-t L^{(i)}} \| \eta \| ^2) (x_0)$. If $P_{t, \omega, \eta, k} ^* \in \Gamma^2 (\mathcal E ^*)^* \simeq \Gamma^2 (\mathcal E)$ denotes the element dual to $P_{t, \omega, \eta, k}$ with respect to the Hermitian product on $\Gamma^2 (\mathcal E ^*)$, we have that
\begin{align*}
\langle \rho_{t, \omega, \eta} ^{(i)}, & P_{t, \omega, \eta, k} \rangle _{\Gamma^2 (\mathcal E ^* | _{\mathcal C_t (\overline {U_i})})} = \int _{\mathcal C_t (\overline {U_i})} \rho_{t, \omega, \eta} ^{(i)} (c) [P_{t, \omega, \eta, k} ^* (c)] \, \mathrm d w_t ^{(i)} (c) = W_{t, \omega, \eta} ^{(i)} (P_{t, \omega, \eta, k} ^*) = \\
& = W_{t, \omega, \eta} ^{(i)} (\chi_k P_{t, \omega, \eta, k} ^*) + W_{t, \omega, \eta} ^{(i)} ((1 - \chi_k) P_{t, \omega, \eta, k} ^*) \ .
\end{align*}

The first term in the right-hand side is
\begin{gather*}
\int _{U_i} \mathrm d x_1 \dots \int _{U_i} \mathrm d x_{2^k} \Bigg\langle \omega \otimes h_\nabla ^{(i)} \left( \frac t {2^k}, x_0, x_1 \right) \otimes \dots \otimes h_\nabla ^{(i)} \left( \frac t {2^k}, x_{2^k-1}, x_{2^k} \right) \otimes \eta(x_{2^k}) , \\
\omega \otimes \chi(x_0, x_1) P(x_0, x_1) \otimes \dots \otimes \chi(x_{2^k-1}, x_{2^k}) P(x_{2^k-1}, x_{2^k}) \otimes \eta (x_{2^k}) \Bigg\rangle = \\
= \left( \frac 1 r \right) ^{2^k + 1} \| \omega \| ^2 _{E_{x_0} ^*} \int _{U_i} \mathrm d x_1 \dots \int _{U_i} \mathrm d x_{2^k} \left\langle h_\nabla ^{(i)} \left( \frac t {2^k}, x_0, x_1 \right), \chi(x_0, x_1) P(x_0, x_1) \right\rangle _{E_{x_0} \otimes E_{x_1}^*} \dots \\
\dots \left\langle h_\nabla ^{(i)} \left( \frac t {2^k}, x_{2^k - 1}, x_{2^k} \right), \chi(x_{2^k - 1}, x_{2^k}) P(x_{2^k - 1}, x_{2^k}) \right\rangle _{E_{x_{2^k-1}} \otimes E_{x_{2^k}}^*} \| \eta (x_{2^k}) \| ^2 _{E_{x_{2^k}}} = \\
= \frac 1 r \| \omega \| ^2 _{E_{x_0} ^*} \left[ \left( R_{\frac t {2^k}} ^{(i)} \right) ^{2^k} \| \eta \| ^2 \right] (x_0) \ ,
\end{gather*}
which converges to $\frac 1 r \| \omega \| ^2 _{E_{x_0} ^*} (\mathrm e ^{-t L^{(i)}} \| \eta \| ^2) (x_0)$ uniformly with respect to $t$ from bounded subsets of $(0, \infty)$ according to theorem \ref{application of Chernoff's theorem}.

Using first the diamagnetic inequality, then the Cauchy-Schwarz inequality in the fiber $\mathcal E_c ^*$, the second term in the above right-hand side may be bounded as such:
\begin{gather*}
|W_{t, \omega, \eta} ^{(i)} ((1 - \chi_k) P_{t, \omega, \eta, k} ^*)| = \left| \int _{\mathcal C_t (\overline {U_i})} \rho_{t, \omega, \eta} ^{(i)} (c) [(1 - \chi_k (c)) \, P_{t, \omega, \eta, k} ^* (c)] \, \mathrm d w_t ^{(i)} (c) \right| = \\
= \left| \int _{U_i} \mathrm d x_1 \dots \int _{U_i} \mathrm d x _{2^k} \, [1 - \chi (x_0, x_1) \dots \chi (x_{2^k - 1}, x_{2^k})] \left\langle \omega \otimes h_\nabla ^{(i)} \left( \frac t {2^k}, x_0, x_1 \right) \otimes \dots \right. \right. \\
\left. \left. \dots \otimes h_\nabla ^{(i)} \left( \frac t {2^k}, x_{2^k - 1}, x_{2^k} \right) \otimes \eta(x_{2^k}), \, \omega \otimes P(x_0, x_1) \otimes \dots \otimes P(x_{2^k - 1}, x_{2^k}) \otimes \eta(x_{2^k}) \right\rangle \right| \le \\
\le \frac 1 r \int _{U_i} \mathrm d x_1 \, h^{(i)} \left( \frac t {2^k}, x_0, x_1 \right) \dots \int _{U_i} \mathrm d x _{2^k} \, h^{(i)} \left( \frac t {2^k}, x_{2^k - 1}, x_{2^k} \right) \\
[1 - \chi (x_0, x_1) \dots \chi (x_{2^k - 1}, x_{2^k})] \| \omega \| _{E_{x_0} ^*} ^2   \| \eta(x_{2^k}) \| _{E_{x_{2^k}}} ^2 = \\
= \frac 1 r \| \omega \| _{E_{x_0} ^*} ^2 \int _{\mathcal C_t (\overline {U_i})} (1 - \chi_k (c)) \| \eta(c(t)) \| _{E_{c(t)}} ^2 \, \mathrm d w_t ^{(i)} (c) = \frac 1 r \| \omega \| _{E_{x_0} ^*} ^2 (\mathrm e ^{-tL^{(i)}} \| \eta \| ^2) (x_0) - \\
- \frac 1 r \| \omega \| _{E_{x_0} ^*} ^2 \int _{U_i} \mathrm d x_1 \, h^{(i)} \left( \frac t {2^k}, x_0, x_1 \right) \chi(x_0, x_1) \dots \\
\dots \int _{U_i} \mathrm d x_{2^k} \, h^{(i)} \left( \frac t {2^k}, x_{2^k - 1}, x_{2^k} \right) \chi(x_{2^k - 1}, x_{2^k}) \| \eta (x_{2^k}) \| ^2 _{E_{x_{2^k}}} = \\
= \frac 1 r \| \omega \| _{E_{x_0} ^*} ^2 (\mathrm e ^{-tL^{(i)}} \| \eta \| ^2) (x_0) - \frac 1 r \| \omega \| _{E_{x_0} ^*} ^2 \left[ \left( S_{\frac t {2^k}} ^{(i)} \right) ^{2^k} \| \eta \| ^2 \right] (x_0) \ ,
\end{gather*}
which converges to $0$ uniformly with respect to $t$ from bounded subsets of $(0, \infty)$ according to corollary \ref{second application of Chernoff's theorem}.

Passing to the limit, we obtain that
\[ \lim _{k \to \infty} \langle \rho_{t, \omega, \eta} ^{(i)}, P_{t, \omega, \eta, k} \rangle _{\Gamma^2 (\mathcal E ^* | _{\mathcal C_t (\overline {U_i})})} = \frac 1 r \| \omega \| ^2 _{E_{x_0} ^*} \lim _{k \to \infty} \left\langle \delta_{x_0}, \left( R_{\frac t {2^k}} ^{(i)} \right) ^{2^k} \| \eta \| ^2 \right\rangle = \frac 1 r \| \omega \| ^2 _{E_{x_0} ^*} (\mathrm e ^{-t L^{(i)}} \| \eta \| ^2) (x_0) \]
uniformly with respect to $t$ from bounded subsets of $(0, \infty)$, whence it follows that
\[ \varlimsup _{k \to \infty} \| \rho_{t, \omega, \eta} ^{(i)} - P_{t, \omega, \eta, k} \| _{\Gamma^2 (\mathcal E ^* | _{\mathcal C_t (\overline {U_i})})} ^2 \le 0 \]
uniformly with respect to $t$ from bounded subsets of $(0, \infty)$, the conclusion being now immediate.
\end{proof}

We already knew that $\rho_{t, \omega, \eta} ^{(i)} \in \Gamma^2 (\mathcal E ^* | _{\mathcal C_t (\overline {U_i})})$, but the convergence that we have just proved allows us to obtain an even stronger conclusion, which will be useful later on, in particular in proving the Feynman-Kac formula in fiber bundles.

\begin{corollary} \label{pointwise norm above regular domains}
$\rho_{t, \omega, \eta} ^{(i)} \in \Gamma^\infty (\mathcal E ^* | _{\mathcal C_t (\overline {U_i})})$ și $\| \rho_{t, \omega, \eta} ^{(i)} (c) \| _{\mathcal E ^* _c} = \frac 1 {\sqrt r} \, \| \omega \| _{E_{x_0}^*} \, \| \eta (c(t)) \| _{E_{c(t)}}$ for almost every $c \in \mathcal C_t (\overline {U_i})$ (with respect to the measure $w_t ^{(i)}$).
\end{corollary}
\begin{proof}
We know that $P_{t, \omega, \eta, k} | _{\mathcal C_t (\overline {U_i})} \to \rho_{t, \omega, \eta} ^{(i)}$ in $\Gamma^2 (\mathcal E ^* | _{\mathcal C_t (\overline {U_i})})$. After choosing a measurable representative of $\rho_{t, \omega, \eta} ^{(i)}$, which we shall denote $\rho_{t, \omega, \eta} ^{(i)}$ again, for simplicity, there exists a subsequence $(k_l) _{l \in \mathbb N} \subseteq \mathbb N$ and a co-null subset $C \subseteq \mathcal C_t (\overline{U_i})$ such that $P_{t, \omega, \eta, k_l} (c) \to \rho_{t, \omega, \eta} ^{(i)} (c)$ for every $c \in C$ (the proof of this fact is almost identical to the proof of the completeness of $\Gamma^2 (\mathcal E)$). Reusing the argument in lemma \ref{approximation of continuous curves}, for each curve $c \in C$ there exists $k_c \in \mathbb N$ such that $P(c(\frac {jt} {2^k}), c(\frac {(j+1)t} {2^k})) \ne 0$ (therefore it is precisely the parallel transport between the two points on $c$) for all $k \ge k_c$ and $0 \le j \le 2^k-1$. It follows that
\[ \| P_{t, \omega, \eta, k} (c) \| _{\mathcal E ^* _c} = \frac 1 {\sqrt r} \, \| \omega \| _{E_{x_0}^*} \, \| \eta (c(t)) \| _{E_{c(t)}} \]
for $c \in C$ and $k \ge k_c$, whence it follows that
\[ \| \rho_{t, \omega, \eta} ^{(i)} (c) \| _{\mathcal E ^* _c} = \lim _{l \to \infty} \| P_{t, \omega, \eta, k_l} (c) \| _{\mathcal E ^* _c} = \frac 1 {\sqrt r} \, \| \omega \| _{E_{x_0}^*} \, \| \eta (c(t)) \| _{E_{c(t)}} \le \frac 1 {\sqrt r} \, \| \omega \| _{E_{x_0}^*} \, \| \eta \| _{\Gamma_{cb} (E)} \ . \]
\end{proof}

So far we have worked on the spaces of curves $\mathcal C_t (\overline{U_i})$ associated to the relatively compact domains $U_i$ that exhaust $M$. We have done this purely for technical reasons, the compactness having to do with the need to bound certain continuous functions in the proof of theorem \ref{application of Chernoff's theorem} that were difficult to control otherwise. It is the appropriate moment now to get rid of this exhaustion and obtain global geometrical objects and global pieces of relationship among these objects.

\begin{theorem}
If $i \le j$ then $\rho_{t, \omega, \eta} ^{(j)} | _{\mathcal C_t (\overline{U_i})} = \rho_{t, \omega, \eta} ^{(i)}$.
\end{theorem}
\begin{proof}
For all $k \in \mathbb N$ we have
\begin{align*}
& \| \rho_{t, \omega, \eta} ^{(j)} | _{\mathcal C_t (\overline{U_i})} - \rho_{t, \omega, \eta} ^{(i)} \| _{\Gamma^2 (\mathcal E ^* | _{\mathcal C_t (\overline{U_i})})} \le \\
& \le \| \rho_{t, \omega, \eta} ^{(j)} | _{\mathcal C_t (\overline{U_i})} - P_{t, \omega, \eta, k} | _{\mathcal C_t (\overline{U_i})} \| _{\Gamma^2 (\mathcal E ^* | _{\mathcal C_t (\overline{U_i})})} + \| P_{t, \omega, \eta, k} | _{\mathcal C_t (\overline{U_i})} - \rho_{t, \omega, \eta} ^{(i)} \| _{\Gamma^2 (\mathcal E ^* | _{\mathcal C_t (\overline{U_i})})} = \\
& = \sqrt{ \int _{\mathcal C_t (\overline{U_i})} \| \rho_{t, \omega, \eta} ^{(j)} (c) - P_{t, \omega, \eta, k} (c) \| _{\mathcal E ^* _c} ^2 \, \mathrm d w_t ^{(i)} (c) } + \| P_{t, \omega, \eta, k} | _{\mathcal C_t (\overline{U_i})} - \rho_{t, \omega, \eta} ^{(i)} \| _{\Gamma^2 (\mathcal E ^* | _{\mathcal C_t (\overline{U_i})})} \le \\
& \le \sqrt{ \int _{\mathcal C_t (\overline{U_i})} \| \rho_{t, \omega, \eta} ^{(j)} (c) - P_{t, \omega, \eta, k} (c) \| _{\mathcal E ^* _c} ^2 \, \mathrm d w_t ^{(j)} | _{\mathcal C_t (\overline{U_i})} (c) } + \| P_{t, \omega, \eta, k} | _{\mathcal C_t (\overline{U_i})} - \rho_{t, \omega, \eta} ^{(i)} \| _{\Gamma^2 (\mathcal E ^* | _{\mathcal C_t (\overline{U_i})})} \le \\
& \le \sqrt{ \int _{\mathcal C_t (\overline{U_j})} \| \rho_{t, \omega, \eta} ^{(j)} (c) - P_{t, \omega, \eta, k} (c) \| _{\mathcal E ^* _c} ^2 \, \mathrm d w_t ^{(j)} (c) } + \| P_{t, \omega, \eta, k} | _{\mathcal C_t (\overline{U_i})} - \rho_{t, \omega, \eta} ^{(i)} \| _{\Gamma^2 (\mathcal E ^* | _{\mathcal C_t (\overline{U_i})})} = \\
& = \| \rho_{t, \omega, \eta} ^{(j)} - P_{t, \omega, \eta, k} | _{\mathcal C_t (\overline{U_j})} \| _{\Gamma^2 (\mathcal E ^* | _{\mathcal C_t (\overline{U_j})})} + \| P_{t, \omega, \eta, k} | _{\mathcal C_t (\overline{U_i})} - \rho_{t, \omega, \eta} ^{(i)} \| _{\Gamma^2 (\mathcal E ^* | _{\mathcal C_t (\overline{U_i})})} \ ,
\end{align*}
whence the conclusion is clear with the aid of theorem \ref{approximation on regular domains}.
\end{proof}

This compatibility relationship among the sections $(\rho _{t, \omega, \eta} ^{(j)}) _{j \in \mathbb N}$ insures that the global section defined by $\rho _{t, \omega, \eta} = \lim _{j \to \infty} \rho _{t, \omega, \eta} ^{(j)}$ is well defined, and that $\rho _{t, \omega, \eta} | _{\mathcal C_t (\overline {U_j})} = \rho _{t, \omega, \eta} ^{(j)}$. In defining $\rho _{t, \omega, \eta}$ as we have done it is understood that we work with measurable representatives of the equivalence classes $\rho _{t, \omega, \eta} ^{(j)} \in \Gamma^\infty (\mathcal E ^* | _{\mathcal C_t (\overline{U_j})}) \subseteq \Gamma^\infty (\mathcal E ^*)$, and that changing these representatives in turn changes the limit only on some null subset, therefore its equivalence class stays the same.

\begin{theorem} \label{pointwise norm}
The section $\rho _{t, \omega, \eta}$ so defined is measurable and essentially bounded. Furthermore, $\| \rho _{t, \omega, \eta} (c) \| _{\mathcal E^* _c} = \frac 1 {\sqrt r} \| \omega \| _{E_{x_0} ^*} \| \eta_{c(t)} \| _{E_{c(t)}}$ for almost all $c \in \mathcal C_t$.
\end{theorem}
\begin{proof}
Since $M = \bigcup _{j \in \mathbb N} U_j$, it follows that $\mathcal C_t = \bigcup _{j \in \mathbb N} \mathcal C_t (\overline {U_j})$. Since $\rho _{t, \omega, \eta} | _{\mathcal C_t (\overline {U_j})} = \rho _{t, \omega, \eta} ^{(j)}$, it follows that if $S \subseteq \mathcal E^*$ is a measurable subset then
\begin{align*}
\rho _{t, \omega, \eta} ^{-1} (S) = \bigcup _{j \in \mathbb N} [\rho _{t, \omega, \eta} ^{-1} (S) \cap \mathcal C_t (\overline {U_j})] = \bigcup _{j \in \mathbb N} [\rho _{t, \omega, \eta} ^{(j)}]^{-1} (S)
\end{align*}
is measurable because each $\rho _{t, \omega, \eta} ^{(j)}$ is measurable.

The value of the pointwise norm of $\rho_{t, \omega, \eta}$ is a consequence of corollary \ref{pointwise norm above regular domains}.
\end{proof}

We have seen in theorem \ref{approximation on regular domains} that $P_{t, \omega, \eta, k} | _{\mathcal C_t (\overline {U_j})} \to \rho _{t, \omega, \eta} | _{\mathcal C_t (\overline {U_j})}$ in $\Gamma^2 (\mathcal E ^* | _{\mathcal C_t (\overline {U_i})})$, for all $j \in \mathbb N$. We shall now prove that it is possible to remove the restriction to $\mathcal C_t (\overline {U_j})$ and obtain the convergence globally, on the whole $\mathcal C_t$.

\begin{theorem} \label{approximation of rho}
The sequence $(P_{t, \omega, \eta, k}) _{k \in \mathbb N}$ converges to $\rho_{t, \omega, \eta}$ în $\Gamma^2 (\mathcal E ^*)$, uniformly with respect to $t \in (0,T]$ for all $T>0$.
\end{theorem}
\begin{proof}
If $\omega = 0$ the result is trivially true; we shall assume then that $\omega \ne 0$. Let $\varepsilon > 0$. Using the fact that $\rho _{t, \omega, \eta} | _{\mathcal C_t (\overline {U_j})} = \rho _{t, \omega, \eta} ^{(j)}$ for all $j \in \mathbb N$, we may write that
\begin{align*}
\| P_{t, \omega, \eta, k} - \rho_{t, \omega, \eta} \| _{\Gamma^2 (\mathcal E ^*)} ^2 & = \| P_{t, \omega, \eta, k} - \rho_{t, \omega, \eta} \| _{\Gamma^2 (\mathcal E ^* | _{\mathcal C_t (\overline {U_j})})} ^2 + \| P_{t, \omega, \eta, k} - \rho_{t, \omega, \eta} \| _{\Gamma^2 (\mathcal E ^* | _{\mathcal C_t \setminus \mathcal C_t (\overline {U_j})})} ^2 \le \\
& \le \| P_{t, \omega, \eta, k} - \rho_{t, \omega, \eta} \| _{\Gamma^2 (\mathcal E ^* | _{\mathcal C_t (\overline {U_j})})} ^2 + \frac 4 r \| \omega \| _{E_{x_0} ^*} ^2 \int \limits _{\mathcal C_t \setminus \mathcal C_t (\overline {U_j})} \| \eta(c(t)) \| _{E_{c(t)}} ^2 \, \mathrm d w_t (c)
\end{align*}
and the integral on the right-hand side is
\begin{align*}
\int _{\mathcal C_t \setminus \mathcal C_t (\overline {U_j})} \| \eta(c(t)) \| _{E_{c(t)}} ^2 \, \mathrm d w_t (c) & = \int _{\mathcal C_t} \| \eta(c(t)) \| _{E_{c(t)}} ^2 \, \mathrm d w_t (c) - \int _{\mathcal C_t (\overline {U_j})} \| \eta(c(t)) \| _{E_{c(t)}} ^2 \, \mathrm d w_t (c) \le \\
& \le \int _{\mathcal C_t} \| \eta(c(t)) \| _{E_{c(t)}} ^2 \, \mathrm d w_t (c) - \int _{\mathcal C_t (\overline {U_j})} \| \eta(c(t)) \| _{E_{c(t)}} ^2 \, \mathrm d w_t ^{(j)} (c) \ .
\end{align*}

On the other hand,
\begin{align*}
\| P_{t, \omega, \eta, k} & - \rho_{t, \omega, \eta} \| _{\Gamma^2 (\mathcal E ^* | _{\mathcal C_t (\overline {U_j})})} ^2 = \\
& = \int _{\mathcal C_t (\overline {U_j})} \| P_{t, \omega, \eta, k} (c) - \rho_{t, \omega, \eta} (c) \| _{\mathcal E ^* _c} ^2 \, \mathrm d w_t ^{(j)} (c) + \\
& + \int _{\mathcal C_t (\overline {U_j})} \| P_{t, \omega, \eta, k} (c) - \rho_{t, \omega, \eta} (c) \| _{\mathcal E ^* _c} ^2 \, \mathrm d (w_t - w_t ^{(j)}) (c) \le \\
& \le \| P_{t, \omega, \eta, k} - \rho_{t, \omega, \eta} \| _{\Gamma^2 (\mathcal E ^* | _{\mathcal C_t (\overline {U_j}), w_t ^{(j)}})} ^2 + \frac 4 r \| \omega \| _{E_{x_0} ^*} ^2 \int _{\mathcal C_t (\overline {U_j})} \| \eta(c(t)) \| _{E_{c(t)}} ^2 \, \mathrm d (w_t - w_t ^{(j)}) (c) \le \\
& \le \| P_{t, \omega, \eta, k} - \rho_{t, \omega, \eta} ^{(j)} \| _{\Gamma^2 (\mathcal E ^* | _{\mathcal C_t (\overline {U_j}), w_t ^{(j)}})} ^2 + \\
& + \frac 4 r \| \omega \| _{E_{x_0} ^*} ^2 \int _{\mathcal C_t} \| \eta(c(t)) \| _{E_{c(t)}} ^2 \, \mathrm d w_t (c) - \frac 4 r \| \omega \| _{E_{x_0} ^*} ^2 \int _{\mathcal C_t (\overline {U_j})} \| \eta(c(t)) \| _{E_{c(t)}} ^2 \, \mathrm d w_t ^{(j)} (c) \ .
\end{align*}

We conclude that
\begin{align*}
\| P_{t, \omega, \eta, k} - \rho_{t, \omega, \eta} \| _{\Gamma^2 (\mathcal E ^*)} ^2 & \le \| P_{t, \omega, \eta, k} - \rho_{t, \omega, \eta} ^{(j)} \| _{\Gamma^2 (\mathcal E ^* | _{\mathcal C_t (\overline {U_j}), w_t ^{(j)}})} ^2 + \\
& + \frac 8 r \| \omega \| _{E_{x_0} ^*} ^2 \left( \int \limits _{\mathcal C_t} \| \eta(c(t)) \| _{E_{c(t)}} ^2 \, \mathrm d w_t (c) - \int \limits _{\mathcal C_t (\overline {U_j})} \| \eta(c(t)) \| _{E_{c(t)}} ^2 \, \mathrm d w_t ^{(j)} (c) \right) .
\end{align*}

We shall choose $j$ by a careful examination of the difference between these two latter integrals:
\begin{gather*}
\int _{\mathcal C_t} \| \eta(c(t)) \| _{E_{c(t)}} ^2 \, \mathrm d w_t (c) - \int _{\mathcal C_t (\overline {U_j})} \| \eta(c(t)) \| _{E_{c(t)}} ^2 \, \mathrm d w_t ^{(j)} (c) = \\
= \int _M h(t, x_0, x) \, \| \eta(x) \| _{E_x} ^2 \, \mathrm d x - \int _{\overline{U_j}} h^{(j)} (t, x_0, x) \, \| \eta(x) \| _{E_x} ^2 \, \mathrm d x \ ,
\end{gather*}
and since $h^{(j)} \to h$ pointwise and monotonically, for each $t>0$ there exists an $j_{\varepsilon, t} \in \mathbb N$ such that
\begin{gather*}
\left| \int _M h(t, x_0, x) \, \| \eta(x) \| _{E_x} ^2 \, \mathrm d x - \int _{U_j} h^{(j)} (t, x_0, x) \, \| \eta(x) \| _{E_x} ^2 \, \mathrm d x \right| = \\
= \left| \langle \delta_{x_0}, \mathrm e ^{-t L} \| \eta \| ^2 \rangle - \langle \delta_{x_0}, \mathrm e ^{-t L^{(j)}} (\| \eta \| | _{\overline {U_j}} ^2) \rangle \right| < \frac {r \varepsilon} {16 \| \omega \| _{E_{x_0} ^*} ^2}
\end{gather*}
for every $j \ge j_{\varepsilon, t}$, where $\langle \cdot, - \rangle$ denotes the duality pairing between the space of bounded continuous functions and its dual (to which $\delta_{x_0}$ belongs). Since the two heat semigroups seen above are strongly continuous, the above expression that contains them is continuous with respect to $t \in [0,T]$, therefore every $t$ from $[0,T]$ admits an open neighbourhood $V _{\varepsilon, t} \subseteq [0,T]$ such that
\[ \left| \langle \delta_{x_0}, \mathrm e ^{-s L} \| \eta \| ^2 \rangle - \langle \delta_{x_0}, \mathrm e ^{-s L^{(j)}} (\| \eta \| | _{\overline {U_j}} ^2) \rangle \right| < \frac {r \varepsilon} {16 \| \omega \| _{E_{x_0} ^*} ^2} \]
for every $j \ge j_{\varepsilon, t}$ uniformly with respect to $s \in V _{\varepsilon, t}$. Since $[0,T]$ is compact, we may cover it with a finite number of such neighbourhoods, $[0,T] = \bigcup _{i=1} ^{N_{\varepsilon}} V_{\varepsilon, t_i}$. Choosing $j_\varepsilon = \max \{ j_{\varepsilon, t_1}, \dots, j_{\varepsilon, t_N} \}$ we have
\[ \left| \langle \delta_{x_0}, \mathrm e ^{-t L} \| \eta \| ^2 \rangle - \langle \delta_{x_0}, \mathrm e ^{-t L^{(j_\varepsilon)}} (\| \eta \| | _{\overline {U_j}} ^2) \rangle \right| < \frac {r \varepsilon} {16 \| \omega \| _{E_{x_0} ^*} ^2} \]
for all $t \in (0,T]$.

Using now theorem \ref {approximation on regular domains} on $U_{j_\varepsilon}$, we may find a $k_\varepsilon \in \mathbb N$ such that
\[ \| P_{t, \omega, \eta, k} - \rho_{t, \omega, \eta} \| _{\Gamma^2 (\mathcal E ^* | _{\mathcal C_t (\overline {U_j}), w_t ^{(j)}})} ^2 < \frac \varepsilon 2 \]
for all $k \ge k_\varepsilon$, uniformly with respect to $t \in (0,T]$.

Combining all these upper bounds we obtain that
\[ \| P_{t, \omega, \eta, k} - \rho_{t, \omega, \eta} \| _{\Gamma^2 (\mathcal E ^*)} ^2 < \varepsilon \]
for all $k \ge k_\varepsilon$, uniformly with respect to $t \in (0,T]$, whence the conclusion is clear.
\end{proof}

\begin{remark}
The section $\rho_{t, \omega, \eta}$ does not depend on the exhaustion with regular domains used to construct it, because it is the limit of the sequence of sections $(P_{t, \omega, \eta, k}) _{k \in \mathbb N}$ that does not depend on any exhaustion.
\end{remark}

\section{Getting rid of the auxiliary section}

Let us remember now that one the aims of this article is to give a new construction of the stochastic parallel transport. Whatever this object may be, it is clear that the stochastic parallel transport of a vector $v \in E_{x_0}$ should not depend on any section $\eta \in \Gamma_{cb} (E)$. Indeed, if one looks back at the proof of theorem \ref{approximation on regular domains}, one sees that $\eta$ was needed only for technical reasons, in order for us to be able to use theorem \ref{application of Chernoff's theorem} and corollary \ref{second application of Chernoff's theorem}, which in turn were based upon Chernoff's theorem. Since $\eta$ is seen to play an exclusively auxiliary role, in the following we shall concentrate our efforts on eliminating it from our results. In order to achieve this, we shall need a number of useful auxiliary results.

Let us begin by showing that $\rho_{t, \omega, \eta}$, which so far has been constructed under the hypothesis that $\eta \in \Gamma_{cb} (E)$, can be extended to a significantly larger class of sections $\eta$. More precisely, let
\[ \Gamma _t ' = \left\{ \eta : M \to E \text{ measurable section } \mid \int _M h(t, x_0, x) \, \| \eta_x \| _{E_x} ^2 \, \mathrm d x < \infty \right\} \]
and let $\Gamma_t$ be the quotient of $\Gamma_t '$ under equality almost everywhere. It is easy to show that $\Gamma_t$ is a Hilbert space, the Hermitian product being
\[ \langle \eta, \eta' \rangle _{\Gamma_t} = \int _M h(t, x_0, x) \, \langle \eta_x, \eta'_x \rangle _{E_x} \, \mathrm d x = \int _{\mathcal C_t} \langle \eta_{c(t)}, \eta'_{c(t)} \rangle _{E_{c(t)}} \, \mathrm d w_t (c) \ . \]
It is also clear that $\Gamma^\infty (E) \subseteq \Gamma_t$.

\begin{theorem}
The space $\Gamma_{cb} (E)$ is dense in $\Gamma_t$.
\end{theorem}
\begin{proof}
It is obvious that $\Gamma_{cb} (E) \subseteq \Gamma_t$. Let $\eta \in \Gamma_{cb} (E) ^\perp$; we shall show that $\eta=0$. If $f \in C_b (M)$ and $\eta' \in \Gamma_{cb} (E)$, then $f \eta' \in \Gamma_{cb} (E) \subseteq \Gamma_t$ and
\[ 0 = \langle f \eta', \eta \rangle _{\Gamma_t} = \int _M f(x) \, \langle \eta'_x, \eta_x \rangle _{E_x} \, \mathrm d x \ . \]
Since $f$ is arbitrary, we conclude that there exists a co-null subset $C_{\eta'} \subseteq M$ such that $\langle \eta'_x, \eta_x \rangle _{E_x} = 0$ for every $x \in C_{\eta'}$ (it is understood that we work with some measurable representative of $\eta$).

Since $M$ is separable, we may cover it with a countable family of trivialization open domains $(V_i)_{i \in \mathbb N}$; by possibly shrinking them we may assume that each $\overline {V_i}$ is a (closed) domain of trivialization. Choose an orthonormal frame in $E | _{\overline {V_i}}$ and use Tietze's theorem to extend it continuously to the whole $M$; let $\{\eta_i ^1, \dots, \eta_i ^r\}$ be the resulting continuous global frame in $E$; the sections making it up will belong to $\Gamma_{cb} (E)$. Fix $i \in \mathbb N$. There exists a co-null $C_{i,j} \subseteq M$ such that $\langle \eta_i ^j (x), \eta_x \rangle _{E_x} = 0$ for all $x \in C_{i,j}$. Letting $C_i = \bigcap _{j=1} ^r C_{i,j} \cap V_i$, we immediately obtain that $\langle u, \eta_x \rangle _{E_x} = 0$ for all $x \in C_i$ and $u \in E_x$, whence $\eta |_{V_i} = 0$ almost everywhere and therefore $\eta = 0$ in $\Gamma_t$.
\end{proof}

If we integrate the result of theorem \ref{pointwise norm} with respect to $c \in \mathcal C_t$, we obtain that
\[ \| \rho_{t, \omega, \eta} \| _{\Gamma^2 (\mathcal E^*)} \le \frac 1 {\sqrt r} \, \| \omega \| _{E_{x_0}} \| \eta \| _{\Gamma _t} \ , \]
when $\eta \in \Gamma_t$; since we have just shown that $\Gamma_{cb} (E)$ is dense in $\Gamma_t$, the map $\Gamma_{cb} (E) \ni \eta \mapsto \rho_{t, \omega, \eta} \in \Gamma^2 (\mathcal E ^*)$ extends to a continuous linear map $\Gamma_t \ni \eta \mapsto \rho_{t, \omega, \eta} \in \Gamma^2 (\mathcal E ^*)$.

\begin{lemma}
If $\eta \in \Gamma_t$ then $\| \rho _{t, \omega, \eta} (c) \| _{\mathcal E ^* _c} = \frac 1 {\sqrt r} \, \| \omega \| _{E_{x_0}} \, \| \eta _{c(t)} \| _{E_{c(t)}}$ for almost all $c \in \mathcal C_t$.
\end{lemma}
\begin{proof}
If $\eta \in \Gamma_t$, let $(\eta_k)_{k \in \mathbb N} \subset \Gamma_{cb} (E)$ be a sequence that converges to $\eta$ in $\Gamma_t$; it follows that $\rho_{t, \omega, \eta_k} \to \rho_{t, \omega, \eta}$ in $\Gamma^2 (\mathcal E^*)$, so there exists a subsequence $(k_l) _{l \in \mathbb N} \subseteq \mathbb N$ such that $\eta_{k_l} \to \eta$ almost everywhere on $M$ and $\rho_{t, \omega, \eta_{k_l}} \to \rho_{t, \omega, \eta}$ almost everywhere on $\mathcal C_t$, whence
\[ \| \rho _{t, \omega, \eta} (c) \| _{\mathcal E ^* _c} = \frac 1 {\sqrt r} \, \| \omega \| _{E_{x_0}} \, \| \eta _{c(t)} \| _{E_{c(t)}} \]
foe almost all $c \in \mathcal C_t$ (again, we have tacitly worked with some arbitrary measurable representative of $\eta$; if we choose another one, this will coincide with the former on a co-null subset of $M$, which does not change the conclusion of the lemma).
\end{proof}

Let $p_t : \mathcal C_t \to M$ be the projection given by $p_t (c) = c(t)$. For every $v \in E_{x_0}$ we shall denote by $v^\flat \in E_{x_0} ^*$ the linear form given by $v^\flat = \sqrt r \, \langle \cdot, v \rangle _{E_{x_0}} \in E_{x_0} ^*$. The notation $p_t ^* E$ will denote the fiber bundle above $\mathcal C_t$ obtained as the pull-back of $E \to M$ under $p_t$. Its fiber $(p_t ^* E) _c$ over the curve $c \in \mathcal C_t$ will be, by definition, $E_{c(t)}$, and we shall use the latter notation for its simplicity. In the following we shall construct, for every $p \in (1, \infty]$, a continuous conjugate-linear map $\Gamma^p (\mathcal E) \ni \xi \mapsto \mathcal P_{t,v} ^p (\xi) \in \Gamma^p (p_t ^* E)$ such that
\[ [\rho_{t, v^\flat, \eta} (c)] \, [\xi(c)] = \langle \mathcal P_{t,v} ^p (\xi) (c), \, \eta(c(t)) \rangle _{E_{c(t)}} \]
for every $\eta \in \Gamma^\infty (E)$.

Let $M = \bigcup _{i \in \mathbb N} V_i '$ be a cover of $M$ with open trivialization domains for $E$. Let $V_0 = V_0 '$ and $V_i = V_i ' \setminus (V_0 \cup \dots \cup V_{i-1})$ for $i \ge 1$; these subsets will be measurable, pairwise disjoint, trivialization domains. Let $\{ \eta _i ^1, \dots, \eta _i ^r \}$ be a measurable orthonormal frame in $E | _{V_i}$. Defining $\eta ^l$ by $\eta ^l | _{V_i} = \eta _i ^l$ for all $1 \le l \le r$ and $i \in \mathbb N$, we obtain a global measurable orthonormal frame $\{ \eta ^1, \dots, \eta ^r \}$ in $E$ made of sections from $\Gamma^\infty (E) \subseteq \Gamma_t$, that is of sections $\eta ^l$ for which $\rho_{t, v^\flat, \eta ^l}$ is a well-defined object as we have seen above. Let $\{ \eta _1, \dots, \eta _r \}$ be the dual frame in $E^*$ defined by $\eta _k (\eta ^l) = \delta _k ^l$ (Kronecker's symbol).

If $\sigma \in \Gamma^{\frac p {p-1}} (p_t ^* E^*)$, then
\[ \sigma(c) = \sum _{l=1} ^r \{\sigma(c)\ \, [\eta ^l (c(t))]\}  \, \eta _l (c(t)) \in E^* _{c(t)} \]
for every $c \in \mathcal C_t$. We define then the functional $\mathcal F_{t,v,\xi} ^p : \Gamma^{\frac p {p-1}} (p_t ^* E^*) \to \mathbb C$ by
\[ \mathcal F _{t,v,\xi} ^p (\sigma) = \sum _{l=1} ^r \int _{\mathcal C_t} \{[\sigma(c)] \, [\eta ^l (c(t))]\} \, \overline{ \{[\rho_{t, v^\flat, \eta ^l} (c)] \, [\xi (c)]\} } \, \mathrm d w_t (c) \]
and it is obvious that it is linear. We have seen above that
\[ \| \rho_{t, v^\flat, \eta ^l} (c) \| _{\mathcal E ^* _c} = \frac 1 {\sqrt r} \, \| v^\flat \| _{E_{x_0}^*} \, \| \eta ^l _{c(t)} \| _{E _{c(t)}} = \| v \| _{E_{x_0}} \ , \]
whence we obtain that
\begin{align*}
|\mathcal F _{t,v,\xi} ^p (\sigma)| & \le \sum _{l=1} ^r \int _{\mathcal C_t} \| \sigma(c) \| _{E _{c(t)} ^*} \, \| \eta_l (c(t)) \| _{E_{c(t)} ^*} \, \| \xi(c) \| _{\mathcal E ^* _c} \, \| \rho_{t, v^\flat, \eta ^l} (c) \| _{\mathcal E^*_c} \, \mathrm d w_t (c) \le \\
& \le r \, \| v \| _{E_{x_0}} \int _{\mathcal C_t} \| \sigma(c) \| _{E _{c(t)} ^*} \, \| \xi(c) \| _{\mathcal E ^* _c} \,  \mathrm d w_t (c) \le \\
& \le r \, \| v \| _{E_{x_0}} \, \| \xi \| _{\Gamma^p (\mathcal E)} \, \| \sigma \| _{\Gamma^{\frac p {p-1}} (p_t ^* E^*)} \ .
\end{align*}
We conclude that there exists a unique section $\mathcal P_{t,v} ^p (\xi) \in \Gamma^p (p_t ^* E)$ such that
\[ \mathcal F _{t,v,\xi} ^p (\sigma) = \int _{\mathcal C_t} [\sigma (c)] \, [\mathcal P _{t,v} ^p (\xi) (c)] \, \mathrm d w_t (c) \]
for all $\sigma \in \Gamma^{\frac p {p-1}} (p_t ^* E^*)$, and that
\[ \| \mathcal P_{t,v} ^p (\xi) \| _{\Gamma^p (p_t ^* E)} \le r \, \| v \| _{E_{x_0}} \, \| \xi \| _{\Gamma^p (\mathcal E)} \ . \]
The continuity and conjugate-linearity of $\xi \mapsto \mathcal P_{t,v} ^p (\xi)$ are obvious.

\begin{corollary}
With the notations above, $\langle \eta_{c(t)}, \, \mathcal P _{t,v} ^p (\xi) (c) \rangle _{E_{c(t)}} = [\rho_{t, v^\flat, \eta} (c)] \, [\xi (c)]$ for all $\eta \in \Gamma^\infty (E)$ and almost all $c \in \mathcal C_t$.
\end{corollary}
\begin{proof}
Let $\eta \in \Gamma^\infty (E)$ and $f \in L^{\frac p {p-1}} (\mathcal C_t)$. Denote by $\eta^\flat \in \Gamma^\infty (E^*)$ the dual section, defined pointwise by $\eta^\flat _x (u) = \langle u, \eta_x \rangle _{E_x}$ for almost all $x \in M$ and all $u \in E_x$. With this notation, $f \eta^\flat \in \Gamma^{\frac p {p-1}} (E^*)$. Using the definitions of $\mathcal P_{t,v} ^p$ and of $\mathcal F _{t,v} ^p$ given above,
\begin{gather*}
\int _{\mathcal C_t} f(c) \, \langle \mathcal P _{t,v} ^p (\xi) (c), \, \eta_{c(t)} \rangle _{E_{c(t)}} \, \mathrm d w_t (c) = \int _{\mathcal C_t} [(f \, p_t ^* \eta^\flat) (c)] \, [\mathcal P _{t,v} ^p (\xi) (c)] \, \mathrm d w_t (c) = \mathcal F _{t,v} ^p (f \, p_t ^* \eta^\flat) = \\
= \sum _{l=1} ^r \int _{\mathcal C_t} \{ [(f \, p_t ^* \eta^\flat) (c)] \, [\eta ^l _{c(t)}] \} \, \overline{ \{[\rho_{t, v^\flat, \eta ^l} (c)] \, [\xi (c)]\} } \, \mathrm d w_t (c) = \\
= \int _{\mathcal C_t} f(c) \sum _{l=1} ^r \langle \eta ^l _{c(t)}, \eta_{c(t)} \rangle _{E_{c(t)}} \overline{ \{[\rho_{t, v^\flat, \eta ^l} (c)] \, [\xi (c)]\} } \, \mathrm d w_t (c) = \\
= \int _{\mathcal C_t} f(c) \, \overline{ \{[\rho_{t, v^\flat, \eta} (c)] \, [\xi (c)]\} } \, \mathrm d w_t (c) \ ,
\end{gather*}
where for the last equality we have used the linearity of the map $\eta \mapsto \rho_{t, v^\flat, \eta}$ and the fact that
\[ \eta_{c(t)} = \sum _{l=1} ^r \langle \eta_{c(t)}, \eta ^l (c(t)) \rangle _{E_{c(t)}} \, \eta ^l (c) \]
for all $c \in \mathcal C_t$.

Since $f$ is arbitrary, we conclude that
\[ \langle \eta_{c(t)}, \, \mathcal P _{t,v} ^p (\xi) (c) \rangle _{E_{c(t)}} = \{[\rho_{t, v^\flat, \eta} (c)] \, [\xi (c)]\} \]
for almost all $c \in \mathcal C_t$.
\end{proof}

The linearity of the map $\mathcal P_{t,v} ^p (\xi)$ with respect to $v \in E_{x_0}$ allows us to define a section $\mathcal P _t ^p (\xi) \in \Gamma^p (p_t ^* E) \otimes E_{x_0} ^*$ by requiring that $\mathcal P _t ^p (\xi) (c) (v) = \mathcal P _{t,v} ^p (\xi) (c)$ for all $v \in E_{x_0}$ and almost all $c \in \mathcal C_t$. Furthermore,
\[ \| \mathcal P_t ^p (\xi) (c) \| _{E_{c(t)} \otimes E_{x_0} ^*} = \sup _{\| v \| _{E_{x_0}} = 1} \| \mathcal P_{t,v} ^p (\xi) (c) \| _{E_{c(t)}} \le r \, \| v \| _{E_{x_0}} \, \| \xi \| _{\Gamma^p (\mathcal E)} = r \, \| \xi \| _{\Gamma^p (\mathcal E)} \ . \]

The map $\mathcal P_t ^p$ encloses a great deal of information regarding the differential geometry and the stochastic calculus associated to the bundle $E$. In the rest of this article we shall see just two of its uses, hopefully enough to convince the reader of its usefulness: the stochastic parallel transport and the Feynman-Kac formula.

\section{The stochastic parallel transport}

Let us begin by defining the sections $\mathcal P_{t,v,k}$ by the explicit formula
\[ \mathcal P_{t,v,k} (c) = P \left( c(t), c \left( \frac {(2^k-1) t} {2^k} \right) \right) \dots P \left( c \left( \frac t {2^k} \right), c(0) \right) v \]
where $v \in E_{x_0}$ is arbitrary, $c \in \mathcal C_t$ and $k \in \mathbb N$. Notice that $\mathcal P_{t,v,k} (c)$ belongs to the fiber $E_{c(t)} = (p_t ^* E) _c$. Since $P$ has been shown to be a measurable map, $\mathcal P_{t,v,k}$ will be a measurable section in $p_t ^* E$. Furthremore, since $\| \mathcal P_{t,v,k} \| _{E_{c(t)}} \le \| v \| _{E_{x_0}}$, we deduce that $\mathcal P_{t,v,k} \in \Gamma^\infty (p_t ^* E) \subseteq \Gamma^2 (p_t ^* E)$.

Let us define the section $\operatorname{Id} : \mathcal C_t \to \mathcal E$ by $\operatorname{Id} (c) = \otimes _{s \in D_t} \operatorname{Id} _{E_{c(s)}} \in \mathcal E _c$; more precisely, $\operatorname{Id} (c)$ is the equivalence class (in the sense of the construction of the algebraic inductive limit as a space of equivalence classes), for instance, of the element $\operatorname{Id} _{E_{x_0}}$, and the map $\mathcal C_t \ni c \mapsto \operatorname{Id} _{E_{x_0}} \in \mathcal E _c$ is obviously continuous. Furthermore, it is obvious that $\| \operatorname{Id} (c) \| _{\mathcal E _c} = \| \operatorname{Id} (E_{x_0}) \| _{\operatorname{End} E_{x_0}} = 1$, so $\operatorname{Id} \in \Gamma^\infty (p_t ^* E) \subseteq \Gamma^2 (p_t ^* E)$. We notice then that
\[ [P_{t, v^\flat, \eta, k} (c)] \, [\operatorname{Id} (c)] = [P_{t, v^\flat, \eta, k} (c)] \, [ \operatorname{Id}_{E_{c(0)}} \otimes \dots \otimes \operatorname{Id}_{E_{c(t)}}] = \langle \eta(c(t)), \, \mathcal P_{t,v,k} (c) \rangle _{E_{c(t)}} \ . \]

\begin{theorem}
$\mathcal P_{t, v, k} \to \mathcal P_{t, v} ^2 (\operatorname{Id})$ in $\Gamma^2 (p_t ^* E)$ for all $v \in E_{x_0}$, uniformly with respect to $t \in (0, T]$ for all $T>0$.
\end{theorem}
\begin{proof}
Using again the global measurbale orthonormal frame $\{ \eta^1, \dots, \eta^r \}$ in $E$,
\begin{align*}
\sup _{t \in (0,T]} \| \mathcal P_{t,v} ^2 (\operatorname{Id}) & - \mathcal P_{t,v,k} \| _{\Gamma^2 (p_t ^* E)} ^2 = \\
& = \sup _{t \in (0,T]} \int _{\mathcal C_t} \| \mathcal P_{t,v} ^2 (\operatorname{Id}) (c) - \mathcal P_{t,v,k} (c) \| _{E _{c(t)}} ^2 \, \mathrm d w_t (c) = \\
& = \sup _{t \in (0,T]} \int _{\mathcal C_t} \sum _{l=1} ^r | \langle \eta ^l (c(t)), \, \mathcal P_{t,v} ^2 (\operatorname{Id}) (c) - \mathcal P_{t,v,k} (c) \rangle _{E_{c(t)}} | ^2 \, \mathrm d w_t (c) = \\
& = \sup _{t \in (0,T]} \sum _{l=1} ^r \int _{\mathcal C_t} | [\rho_{t, v^\flat, \eta ^l} (c) - P_{t, v^\flat, \eta ^l, k} (c)] \, [\operatorname{Id} (c)] | ^2 \, \mathrm d w_t (c) \le \\
& \le \sum _{l=1} ^r \sup _{t \in (0,T]} \int _{\mathcal C_t} \| \rho_{t, v^\flat, \eta ^l} (c) - P_{t, v^\flat, \eta ^l, k} (c) \| _{\mathcal E^* _c} ^2 \, \mathrm d w_t (c) \le \\
& \le \sum _{l=1} ^r \sup _{t \in (0,T]} \| \rho_{t, v^\flat, \eta ^l} - P_{t, v^\flat, \eta ^l, k} \| _{\Gamma^2 (\mathcal E^*)} ^2 \to 0 \ ,
\end{align*}
which together with theorem \ref{approximation of rho} shows the desired convergence.
\end{proof}

Comparing this result with the one obtained by probabilistic techniques (\cite{Ito63}, \cite{Ito75a}, \cite{Ito75b}), we conclude that $\mathcal P_{t, v} ^2 (\operatorname{Id})$ is the stochastic parallel transport in $E$ of the vector $v \in E_{x_0}$. In particular, $\mathcal P _{t,v} ^2 (\operatorname{Id})$ does not depend on the choices made in its construction (the domains of trivialization, the orthonormal frames above them etc.), being the limit of a sequence of sections that do not depend on these choices.

\begin{corollary}
$\| \mathcal P _{t,v} ^2 (\operatorname{Id}) (c) \| _{E_{c(t)}} = \| v \| _{E_{x_0}}$ for almost every curve $c \in \mathcal C_t$.
\end{corollary}
\begin{proof}
Since $\mathcal P_{t, v, k} \to \mathcal P_{t,v} ^2 (\operatorname{Id})$ in $\Gamma^2 (p_t ^* E)$, there exists a subsequence $(k_j) _{j \in \mathbb N} \subseteq \mathbb N$ such that $\mathcal P_{t, v, k_j} (c) \to \mathcal P_{t, v} ^2 (\operatorname{Id}) (c)$ in $E_{c(t)}$ for almost all $c \in \mathcal C_t$. Let $c$ be such a curve; using again the argument in lemma \ref{pointwise norm}, there exists a $l_c \in \mathbb N$ such that $\mathcal P_{t, v, k} (c)$ is the parallel transport of $v$ along a zig-zag line made of $2^k-1$ minimizing geodesic segments, for all $k \ge l_c$, so
\[ \| \mathcal P_{t, v} ^2 (\operatorname{Id}) (c) \| _{E_{c(t)}} = \lim _{j \to \infty } \| \mathcal P_{t, v, k_j} (c) \| _{E_{c(t)}} = \| v \| _{E_{x_0}} \ . \]
\end{proof}

Since $\| \mathcal P _t ^2 (\operatorname{Id}) (v) (c) \| _{E_{c(t)}} = \| v \| _{E_{x_0}}$ for almost every curve $c \in \mathcal C_t$, it makes sense to talk about $\mathcal P _t ^2 (\operatorname{Id}) ^{-1}$. One sees immediately that this object will be a section in $p_t ^* E$ with values in $E_{x_0}$, or formally $\mathcal P _t ^2 (\operatorname{Id}) ^{-1} \in \Gamma^2 (p_t ^* E) \otimes E_{x_0}$.

\section{The Feynman-Kac formula in vector bundles}

In the following we shall state and prove an extension in Hermitian bundles of the Feynman-Kac formula. Consider a "potential" $V \in \Gamma^1 _{loc} (\operatorname{End} E)$ with the property $\operatorname{ess \, inf} _{x \in M} \min \operatorname{spec} V(x) = \beta > -\infty$ (for short: $V \ge \beta$), and with $V(x)$ self-adjoint for almost all $x \in M$. The quadratic form $\Gamma_0 (E) \ni\eta \mapsto \int _M \langle V(x) \eta_x, \eta_x \rangle _{E_x} \, \mathrm d x \in \mathbb R$ will give rise to a densely-defined self-adjoint operator in $\Gamma^2 (E)$, that we shall denote again by $V$, for simplicity. Indeed, the quadratic form is well-defined because
\[ \left| \int _M \langle V(x) \eta_x, \eta_x \rangle _{E_x} \, \mathrm d x \right| \le \sup _{x \in M} \| \eta_x \| ^2 \int _{\operatorname{supp} \eta} \| V(x) \| _{op} \, \mathrm d x < \infty \ . \]
It is also lower-bounded by $\beta$ because, if $\{ e_{1,x}, \dots, e_{r,x} \}$ is an orthonormal basis in $E_x$ made of eigenvectors of $V(x)$ with corresponding eigenvalues $\lambda_{1,x} \le \dots \le \lambda_{r,x} \subset [\beta, \infty)$ for every $x \in M$, and if $\eta_x = \sum _{i=1} ^r \alpha_{i,x} \, e_{i,x}$, we have that
\begin{align*}
\langle V(x) \eta_x, \eta_x \rangle _{E_x} & = \left< \sum _{i=1} ^r \alpha_{i,x} \lambda_{i,x} e_{i,x}, \sum _{j=1} ^r \alpha_{j,x} e_{j,x} \right> _{E_x} = \\
& = \sum _{i=1} ^r \lambda_{i,x} | \alpha_{i,x} | ^2 \ge \sum _{i=1} ^r \lambda_{1,x} | \alpha_{i,x} | ^2 = \lambda_{1,x} \| \eta \| ^2 _{E_x} \ge \beta \| \eta \| ^2 _{E_x} \ .
\end{align*}
One may construct the self-adjoint, densely-defined operator corresponding to the sum of $H_\nabla$ and $V$ in the same way, using quadratic forms. Of course, the same construction may be performed not only on $M$, but also on any relatively compact open subset with smooth boundary.

When the starting point of the continuous curves is no longer the fixed point $x_0 \in M$, as until now, but some variable $x \in M$, all the objects that depend on it will gain it as a supplementary lower index; this means that the space of continuous curves starting at $x$ will be $\mathcal C _{t,x}$, on which we shall have the Wiener measure $w_{t,x}$, and all the objects constructed ini this article so far will also gain a supplementary lower index $x$, meaning that we shall have the sections $\rho_{t, \omega, \eta, x}$, $\mathcal P _{t,v,x}$ and $\mathcal P_{t,x}$ etc.

For each $k \in \mathbb N$ denote by $\operatorname V_{t,x,k} \in \Gamma^\infty (\mathcal E)$ the section given by
\[ \operatorname V_{t,x,k} (c) = \mathrm e ^{- \frac t {2^k} V \left( c \left( \frac {t} {2^k} \right) \right)} \otimes \dots \otimes \mathrm e ^{- \frac t {2^k} V(c (t))} \ . \]
Since $V \ge \beta$ și $t \ge 0$, it is immediate that $\| \operatorname V_{t,x,k} (c) \| _{\mathcal E _c} \le \mathrm e^{-t \beta}$ for almost all $c \in \mathcal C_{t,x}$, whence we conclude with the Banach-Alaoglu theorem that there xists a subsequence $(k_l) _{l \in \mathbb N} \subseteq \mathbb N$ such that the subsequence $(\operatorname V_{t,x,k_l}) _{l \in \mathbb N}$ has a weak limit denoted $\operatorname V_{t,x} \in \Gamma^2 (\mathcal E)$. In particular, we conclude that the section $\mathcal P _{t,x} ^2 (\operatorname V_{t,x})$ exists in $\Gamma^2 (p_t ^* E) \otimes E_x ^*$.

\begin{theorem}[The Feynman-Kac formula]
If $\eta \in \Gamma^2(E)$ then
\[ (\mathrm e ^{-t H_\nabla - t V} \eta) (x) = \int _{\mathcal C_{t,x}} [\mathcal P_{t,x} ^2 (\operatorname V_{t, x}) (c)] ^* \, \eta_{c(t)} \, \mathrm d w_{t,x} (c) \]
for every $t>0$ and almost all $x \in M$.
\end{theorem}
\begin{proof}
Let us consider an exhaustion $M = \bigcup _{j \ge 0} U_j$ with relatively compact connected open subsets with smooth boundary, as we have already done in this article, the notations being the ones already encountered. For every $x \in M$ there exists a $j_x \in \mathbb N$ such that $x \in U_j$ for all $j \ge j_x$. This means that for every $x \in M$ it makes sense to consider the spaces $\mathcal C_{t,x} (\overline{U_j})$ for large enough $j$, and this is enough because in the following we shall let $j \to \infty$.

From theorem 4.2 in \cite{Simon78} applied to the exponential function we have that
\[ \mathrm e ^{-t H_\nabla -t V} = \lim _{j \to \infty} \mathrm e ^{-t H_\nabla ^{(j)} -t V} \]
strongly in $\Gamma^2(E)$, while from the Trotter-Kato formula (see \cite{Kato74}) we have that 
\[ \mathrm e ^{-t H_\nabla ^{(j)} -t V} = \lim _{l \to \infty} \left( \mathrm e ^{-\frac t {2^{k_l}} L_\nabla ^{(j)}} \, \mathrm e ^{-\frac t {2^{k_l}} V} \right) ^{2^{k_l}} \]
strongly in $\Gamma^2(E |_{U_j})$, where $(k_l) _{l \in \mathbb N} \subseteq \mathbb N$ is the subsequence found right above this theorem. It follows that there exists a sub-subsequence $(k_{l_m}) _{m \in \mathbb N} \subseteq \mathbb N$ such that
\[ [\mathrm e ^{-t H_\nabla ^{(j)} - t V} \, \eta] (x) = \lim _{m \to \infty} \left( \mathrm e ^{-\frac t {2^{k_{l_m}}} H_\nabla ^{(j)}} \, \mathrm e ^{-\frac t {2^{k_{l_m}}} V} \, \eta \right) ^{2^{k_{l_m}}} (x) \]
for all $\eta \in \Gamma^2 (E)$ and almost all $x \in U_j$.

It follows that if $\eta, \eta' \in \Gamma^2(E)$, then
\begin{gather}
\nonumber \langle \mathrm e ^{-t H_\nabla - t V} \eta, \eta' \rangle _{\Gamma^2(E)} = \lim _{j \to \infty} \langle \mathrm e ^{-t H_\nabla ^{(j)} -t V} \eta, \eta' \rangle _{\Gamma^2 (E | _{U_j})} = \\
\nonumber = \lim _{j \to \infty} \left< \lim _{m \to \infty} \left( \mathrm e ^{-\frac t {2^{k_{l_m}}} H_\nabla ^{(j)}} \, \mathrm e ^{-\frac t {2^{k_{l_m}}} V} \right) ^{2^{k_{l_m}}} \eta, \eta' \right> _{\Gamma^2 (E | _{U_j})} = \\
\nonumber = \lim _{j \to \infty} \int _{U_j} \mathrm d x \lim _{m \to \infty} \left< \int _{U_j} \mathrm d x_1 \, h_\nabla ^{(j)} \left( \frac t {2^{k_{l_m}}}, x, x_1 \right) \mathrm e ^{-\frac t {2^{k_{l_m}}} V(x_1)} \right. \dots \\
\nonumber \dots \left. \int _{U_j} \mathrm d x_{2^{k_{l_m}}} \, h_\nabla ^{(j)} \left( \frac t {2^{k_{l_m}}}, x_{2^{k_{l_m}} - 1}, x_{2^{k_{l_m}}} \right) \mathrm e ^{-\frac t {2^{k_{l_m}}} V(x_{2^{k_{l_m}}})} \eta(x_{2^{k_{l_m}}}) , \eta'_x \right> _{E_x} = \\
\nonumber = \lim _{j \to \infty} \int _{U_j} \mathrm d x \lim _{m \to \infty}  W_{t, {\eta'_x} ^\flat, \eta, x} ^{(j)} \left( \operatorname V_{t, x, k_{l_m}} | _{\mathcal C_{t,x} (\overline{U_j})} \right) = \\
\nonumber = \lim _{j \to \infty} \int _{U_j} \mathrm d x \lim _{m \to \infty} \int _{\mathcal C_{t,x} (\overline{U_j})} [\rho_{t, {\eta'_x} ^\flat, \eta, x} ^{(j)} (c)] \, [\operatorname V_{t, x, k_{l_m}} (c)] \, \mathrm d w_{t,x} ^{(j)} (c) = \\
= \lim _{j \to \infty} \int _{U_j} \mathrm d x \lim _{m \to \infty} \int _{\mathcal C_{t,x} (\overline{U_j})} [\rho_{t, {\eta'_x} ^\flat, \eta, x} ^{(j)} (c)] \, [\operatorname V_{t, x, k_{l_m}} (c)] \, \mathrm d w_{t,x} (c) + \label{first summand} \\
+ \lim _{j \to \infty} \int _{U_j} \mathrm d x \lim _{m \to \infty} \int _{\mathcal C_{t,x} (\overline{U_j})} [\rho_{t, {\eta'_x} ^\flat, \eta, x} ^{(j)} (c)] \, [\operatorname V_{t, x, k_{l_m}} (c)] \, \mathrm d [w_{t,x} ^{(j)} - w_{t,x}] (c) \ . \label{second summand}
\end{gather}

Due to the weak convergence of $\operatorname V_{t,x,{k_{l_m}}}$ to $\operatorname V_{t,x}$ in $\Gamma^2 (\mathcal E)$, and therefore in $\Gamma^2 (\mathcal E | _{\mathcal C_{t,x} (\overline{U_j})})$ (both spaces being considered with respect to the measure $w_{t,x}$), the term (\ref{first summand}) is
\begin{gather*}
\lim _{j \to \infty} \int _{U_j} \mathrm d x \int _{\mathcal C_{t,x} (\overline{U_j})} [\rho_{t, {\eta'_x} ^\flat, \eta, x} ^{(j)} (c)] \, [\operatorname V_{t, x} (c)] \, \mathrm d w_{t,x} (c) = \\
= \int _M \mathrm d x \int _{\mathcal C_{t,x}} [\rho_{t, {\eta'_x} ^\flat, \eta, x} (c)] \, [\operatorname V_{t, x} (c)] \, \mathrm d w_{t,x} (c) = \\
= \int _M \mathrm d x \int _{\mathcal C_{t,x}} \langle \eta_{c(t)}, \mathcal P_{t,x} ^2 (\operatorname V_{t, x}) (c) \, \eta'_x \rangle _{E_{c(t)}} \, \mathrm d w_{t,x} (c) = \\
= \int _M \mathrm d x \int _{\mathcal C_{t,x}} \langle [\mathcal P_{t,x} ^2 (\operatorname V_{t, x}) (c)] ^* \, \eta_{c(t)}, \eta'_x \rangle _{E_{c(t)}} \, \mathrm d w_{t,x} (c) = \\
= \int _M \mathrm d x \left\langle \int _{\mathcal C_{t,x}} [\mathcal P_{t,x} ^2 (\operatorname V_{t, x}) (c)] ^* \, \eta_{c(t)} \, \mathrm d w_{t,x} (c), \eta'_x \right\rangle _{E_x} \ .
\end{gather*}

In order to obtain the limit when $j \to \infty$ we have applied the dominated convergence theorem on $M$ to the limit
\[ \lim _{j \to \infty} 1 _{U_j} (x) \int _{\mathcal C_{t,x} (\overline{U_j})} [\rho_{t, {\eta'_x} ^\flat, \eta, x} ^{(j)} (c)] [\operatorname V_{t,x,{k_l}} (c)] \, \mathrm d w_{t,x} (c) = \int _{\mathcal C_{t,x}} [\rho_{t, {\eta'_x} ^\flat, \eta, x} (c)] [\operatorname V_{t,x,{k_l}} (c)] \, \mathrm d w_{t,x} (c) \]
valid for almost all $x \in M$, where $1_{U_j}$ is the characteristic function of $U_j$. The domination is insured by the fact that both $\| \rho_{t, {\eta'_x} ^\flat, \eta, x} ^{(j)} (c) \| _{\mathcal E ^* _c}$ and $\| \rho_{t, {\eta'_x} ^\flat, \eta, x} (c) \| _{\mathcal E ^* _c}$ are bounded by $\| \eta'_x \| _{E_x} \, \| \eta (c(t)) \| _{E_{c(t)}}$, and $\| \operatorname V_{t,x,{k_l}} (c) \| _{\mathcal E _c}$ is bounded by $\mathrm e ^{-t \beta}$, for almost all $c$, hence
\begin{gather*}
\left| 1 _{U_j} (x) \int _{\mathcal C_{t,x} (\overline{U_j})} [\rho_{t, {\eta'_x} ^\flat, \eta, x} ^{(j)} (c)] [\operatorname V_{t,x,{k_l}} (c)] \, \mathrm d w_{t,x} (c) \right| \le \mathrm e ^{-t \beta} \, \| \eta'_x \| _{E_x} \int _{\mathcal C_{t,x} (\overline{U_j})}  \| \eta (c(t)) \| _{E_{c(t)}} \, \mathrm d w_{t,x} (c) \le \\
\le \mathrm e ^{-t \beta} \, \| \eta'_x \| _{E_x} \int _{\mathcal C_{t,x}}  \| \eta (c(t)) \| _{E_{c(t)}} \, \mathrm d w_{t,x} (c) = \mathrm e ^{-t \beta} \, \| \eta' _x \| _{E_x} \int _M h(t,x,y) \, \| \eta_y \| _{E_y} \, \mathrm d y \le \\
\le \mathrm e ^{-t \beta} \, \| \eta' _x \| _{E_x} \, (\mathrm e ^{-t H_{\mathrm d, 0}} \| \eta \|) (x) \ ,
\end{gather*}
and the latter function is finite at every $x \in M$ and has the integral
\[ \int _M \| \eta' _x \| _{E_x} \, (\mathrm e ^{-t H} \| \eta \|) (x) \, \mathrm d x = \langle \| \eta' \|, \, \mathrm e ^{-t H} \| \eta \| \rangle _{L^2 (M)} \le \| \eta' \| _{\Gamma^2 (E)} \, \| \eta \| _{\Gamma^2 (E)} < \infty \ . \]

Using the same majorizations as above, and using that $w_{t,x} ^{(j)} \le w_{t,x}$, the term (\ref{second summand}) is $0$ because
\begin{gather*}
\lim _{j \to \infty} \left| \int _{U_j} \mathrm d x \lim _{m \to \infty} \int _{\mathcal C_{t,x} (\overline{U_j})} [\rho_{t, {\eta'_x} ^\flat, \eta, x} ^{(j)} (c)] \, [\operatorname V_{t, x, k_{l_m}} (c)] \, \mathrm d [w_{t,x} ^{(j)} - w_{t,x}] (c) \right| \le \\
\le \lim _{j \to \infty} \int _{U_j} \mathrm d x \lim _{m \to \infty} \int _{\mathcal C_{t,x} (\overline{U_j})} \| \rho_{t, {\eta'_x} ^\flat, \eta, x} ^{(j)} (c) \| _{\mathcal E ^* _c} \, \| \operatorname V_{t, x, k_{l_m}} (c) \| _{\mathcal E _c} \, \mathrm d [w_{t,x} - w_{t,x} ^{(j)}] (c) \le \\
\le \mathrm e ^{-t \beta} \lim _{j \to \infty} \int _{U_j} \mathrm d x \, \| \eta'_x \| _{E_x} \int _{\mathcal C_{t,x} (\overline{U_j})} \| \eta_{c(t)} \| _{E_{c(t)}} \, \mathrm d [w_{t,x} - w_{t,x} ^{(j)}] (c) \le \\
\le \mathrm e ^{-t \beta} \lim _{j \to \infty} \int _{U_j} \mathrm d x \, \| \eta'_x \| _{E_x} \left[ \int _{\mathcal C_{t,x}} \| \eta_{c(t)} \| _{E_{c(t)}} \, \mathrm d w_{t,x} (c) - \int _{\mathcal C_{t,x} (\overline{U_j})} \| \eta_{c(t)} \| _{E_{c(t)}} \, \mathrm d w_{t,x} ^{(j)} (c) \right] \le \\
\le \| \eta' \| _{\Gamma^2 (E)} \, \lim _{j \to \infty} \left\| \mathrm e ^{-t H} \| \eta \| - \mathrm e ^{-t H ^{(j)}} \| \eta \| \right\| _{L^2 (M)} = 0 \ .
\end{gather*}

We conclude that
\[ \langle \mathrm e ^{-t H_\nabla - t V} \eta, \eta' \rangle _{\Gamma^2(E)} = \int _M \mathrm d x \left\langle \int _{\mathcal C_{t,x}} [\mathcal P_{t,x} ^2 (\operatorname V_{t, x}) (c)] ^* \, \eta_{c(t)} \, \mathrm d w_{t,x} (c), \eta'_x \right\rangle _{E_x} \ , \]
whence
\[ (\mathrm e ^{-t H_\nabla - t V} \eta) (x) = \int _{\mathcal C_{t,x}} [\mathcal P_{t,x} ^2 (\operatorname V_{t, x}) (c)] ^* \, \eta_{c(t)} \, \mathrm d w_{t,x} (c) \]
for every $\eta \in \Gamma^2(E)$ and almost all $x \in M$.
\end{proof}

Notice that if we define the map $\mathcal V_{t,x} : \mathcal C_{t,x} \to \operatorname{End} E_x$ by
\[ \mathcal V _{t,x} (c) = [\mathcal P_{t,x} ^2 (\operatorname V_{t, x}) (c)] ^* \, [\mathcal P_{t,x} ^2 (\operatorname {Id}) (c)] \]
we may trivially rewrite the Feynman-Kac formula under the equivalent form
\[ (\mathrm e ^{-t H_\nabla - t V} \eta) (x) = \int _{\mathcal C_{t,x}} [\mathcal V_{t,x} (c)] \, [\mathcal P_{t,x} ^2 (\operatorname {Id}) (c)] ^{-1} \, \eta_{c(t)} \, \mathrm d w_{t,x} (c) \ ; \]
so rewritten, the Feynman-Kac formula in bundles has been obtained by other authors, too, but in other contexts and under different assumptions:
\begin{itemize}[wide]
\item the authors of \cite{BP08} use functional-analytic techniques (again based on Chernoff's theorem) but the potential $V$ is assumed smooth and $M$ is a closed manifold;
\item the authors of \cite{DT01} use probabilistic techniques to give abstract conditions in proposition 4.5 under which the Feynman-Kac formula is valid, after which proposition 5.1 shows that these conditions are met when $M$ is closed; the potential (denoted therein by $\mathcal R$) is assumed smooth (p.48); 
\item in \cite{Guneysu10} the Feynman-Kac formula is proved using functional-analytic techniques, but assuming the existence of the stochastic parallel transport, under the assumption that the manifold is both metrically and stochastically complete, and under very generous hypotheses on the potential (in theorem 3.1 it is assumed essentially bounded, and in theorem 3.3 the result is extended to the more general situation when the potential is locally square-integrable); in remark 1.4 therein the author sketches the modifications to be made to the proof in order for the assumption of metric completeness to be dropped, but does not give further details;
\item in \cite{BG20} (arXiv preprint, submitted for publication but still unpublished at the date of writing of this text), the potential $V$, which may be understood as a differential operator of order $0$, is assumed now to be a differential operator of order $0$ or $1$ acting on the smooth sections in $E$ (in particular, this means that $V$ has smooth coefficients), such that the operator $\nabla ^* \nabla + V$ should be sectorial; such a potential gives rise naturally to a stochastic differential equation, the unique solution of which is assumed to be locally square-integrable in a certain uniform way with respect to $x \in M$ (in our notations); this hypothesis guarantees that the equality in the Feynman-Kac formula will be valid everywhere, not just almost everywhere. No restrictions are placed on the manifold $M$.
\end{itemize}

It can be seen, by way of comparison with the cited previous work, that the Feynman-Kac formula presented here seems to be the most general one currently existing in the literature.

\begin{corollary}
If $V : M \to \mathbb R$ is continuous and lower-bounded, then the above Feynman-Kac formula reduces to
\[ (\mathrm e ^{-t H_\nabla -t V} \eta) (x) = \int _{\mathcal C_{t,x}} \mathrm e ^{- \int _0 ^t V(c(s)) \, \mathrm d s} \, [\mathcal P_{t,x} ^2 (\operatorname {Id}) (c)] ^{-1} \, \eta_{c(t)} \, \mathrm d w_{t,x} (c) \ . \]
\end{corollary}
\begin{proof}
When $V$ is a continuous scalar function,
\begin{align*}
\operatorname V_{t,x,k} (c) & = \mathrm e ^{- \frac t {2^k} V \left( c \left( \frac {t} {2^k} \right) \right)} \otimes \dots \otimes \mathrm e ^{- \frac t {2^k} V(c (t))} = \mathrm e ^{- \frac t {2^k} \sum _{j=1} ^{2^k} V \left( c \left( \frac {jt} {2^k} \right) \right)} \, \operatorname{Id} _{E_{c(\frac t {2^k})}} \otimes \dots \otimes \operatorname{Id} _{E_{c(\frac {2^k t} {2^k})}} = \\
& = \mathrm e ^{- \frac t {2^k} \sum _{j=1} ^{2^k} V \left( c \left( \frac {jt} {2^k} \right) \right)} \, \operatorname{Id} \to \mathrm e ^{- \int _0 ^t V(c(s)) \, \mathrm d s} \, \operatorname{Id} \ ,
\end{align*}
the convergence being valid for all $c \in \mathcal C_{t,x}$, therefore also weakly in $\Gamma^2 (\mathcal E)$, using the dominated convergence theorem. It follows that
\[ \mathcal P _{t,x} ^2 (\operatorname V_{t,x}) = \mathrm e ^{- \int _0 ^t V(c(s)) \, \mathrm d s} \, \mathcal P _{t,x} ^2 (\operatorname {Id}) \]
and the conclusion is immediate.
\end{proof}

\begin{remark}
Comparing the results herein with the ones obtained by the author in \cite{Mustatea22}, we notice that if $E = M \times \mathbb C$ and $\nabla = \mathrm d + \mathrm i \alpha$, then
\[ \mathcal P_{t,x} ^2 (\operatorname{Id}) = \mathrm e^{- \mathrm i \operatorname{Strat} _{t,x} (\alpha)} \]
for almost all $x \in M$. This shows once more that the Stratonovich stochastic integral is the "most geometrically-flavoured" of all the stochastic integrals considered therein, since its exponential (including the negative imaginary unit) is the stochastic parallel transport, in perfect analogy with how the parallel transport along some smooth curve $c$ with respect to $\nabla$ is $\mathrm e ^{- \mathrm i \int _c \alpha}$.
\end{remark}

When $V=0$, the Feynman-Kac formula and the disintegration theorem for measures allow us to derive a formula expressing the heat kernel $h_\nabla$ in the bundle $E$ in terms of the heat kernel $h$ acting on functions and the stochastic parallel transport in $E$. In order to state it, we shall need to introduce some notations.

Let us endow the manifold $M$ with the measure $h(t, x, \cdot) \, \mathrm d x$, where $\mathrm d x$ is the natural measure on $M$. The map $p_t : \mathcal C_{t,x} \to M$ satisfies the hypotheses of the disintegration theorem (p.78-III and following of \cite{DM78}), therefore there exists a family $(\nu_{t,x,y}) _{y \in M}$ of Borel regular probabilities on $\mathcal C_{t,x}$, uniquely determined for almost all $y \in M$, such that $\nu_{t,x,y}$ is concentrated on $p_t ^{-1} (\{y\}) = \{ c \in \mathcal C_{t,x} \mid c(t) = y \}$ for almost all $y \in M$, and
\[ \int _{\mathcal C_{t,x}} f \, \mathrm d w_{t,x} = \int _M h(t, x, y) \left( \int _{p_t ^{-1} (\{y\})} f \, \mathrm d \nu_{t,x,y} \right) \mathrm d y \]
for all $f \in L^1 (\mathcal C_{t,x})$.

\begin{corollary}
\[ h_\nabla (t,x,y) = h(t,x,y) \int_{p_t ^{-1} (\{y\})} [\mathcal P_{t,x} ^2 (\operatorname{Id})] ^{-1} \, \mathrm d \nu_{t,x,y} \in E_x \otimes E_y ^* \]
for all $t>0$, all $x \in M$ and almost all $y \in M$.
\end{corollary}
\begin{proof}
Choosing $V=0$ in the Feynman-Kac formula in bundles, we obtain
\begin{gather*}
\omega \left[ \int _M h_\nabla (t,x,y) \, \eta_y \, \mathrm d y \right] = \omega \{ [\mathrm e ^{-t H_\nabla} \eta] (x) \} =  \omega \left[ \int _{\mathcal C_{t,x}} [\mathcal P_{t,x} ^2 (\operatorname {Id}) (c)] ^{-1} \, \eta_{c(t)} \, \mathrm d w_{t,x} (c) \right] = \\
= \omega \left[ \int _M \mathrm d y \, h(t,x,y) \int _{p_t ^{-1} (\{y\})} [\mathcal P_{t,x} ^2 (\operatorname {Id}) (c)] ^{-1} \, \eta_y \, \mathrm d \nu_{t,x,y} (c) \right]
\end{gather*}
for all $\eta \in \Gamma^2 (E)$ and $\omega \in E_x ^*$, whence the conclusion is clear.
\end{proof}

\begin{remark}
Since $h_\nabla (t,x,y) ^* = h_\nabla (t,y,x)$, the above equality may be rewritten, equivalently, as
\[ h_\nabla (t,x,y) = h(t,x,y) \int_{p_t ^{-1} (\{x\})} [\mathcal P_{t,y} ^2 (\operatorname{Id})] \, \mathrm d \nu_{t,y,x} \in E_x \otimes E_y ^* \]
for all $t>0$ and $y \in M$, and for almost all $x \in M$.
\end{remark}

If $M = \bigcup _{j \in \mathbb N} U_j$, we already know that $h_j \to h$ pointwise; as a final application of all the results obtained above, we shall show that $h_\nabla ^{(j)} \to h_\nabla$ pointwise, too. It is clear that the disintegration theorem may be used, analogously, on each space $\mathcal C_{t,x} (\overline{U_j})$ endowed with the measure $w_{t,x} ^{(j)}$, in order to obtain
\[ \int _{\mathcal C_{t,x} (\overline{U_j})} f \, \mathrm d w_{t,x} ^{(j)} = \int _{U_j} h ^{(j)} (t, x, y) \left( \int _{p_t ^{-1} (\{y\})} f \, \mathrm d \nu_{t,x,y} ^{(j)} \right) \mathrm d y \]
for all $f \in L^1 (\mathcal C_{t,x} (\overline{U_j}), w_{t,x} ^{(j)}) \subseteq L^1 (\mathcal C_{t,x}, w_{t,x})$.

\begin{lemma}
\[ \lim _{j \to \infty} \int _{p_t ^{-1} (\{y\})} f \, \mathrm d \nu_{t,x,y} ^{(j)} = \int _{p_t ^{-1} (\{y\})} f \, \mathrm d \nu_{t,x,y} \]
for all $f \in L^1 (\mathcal C_{t,x})$, all $t>0$ and $x \in M$, and almost all $y \in M$.
\end{lemma}

\begin{corollary}
$h_\nabla (t,x,y) = \lim _{j \to \infty} h_\nabla ^{(j)} (t,x,y)$ for all $t>0$ and $x,y \in M$.
\end{corollary}
\begin{proof}
We know that $h(t,x,y) = \lim _{j \to \infty} h^{(j)} (t,x,y)$ for all $t>0$ și $x,y \in M$, whence, by corroborating the preceding results, we obtain that there exists a co-null $C \subseteq M$ such that $h_\nabla (t,x,y) = \lim _{j \to \infty} h_\nabla ^{(j)} (t,x,y)$ for all $t>0$, $x \in M$ and all $y \in C$. Since $h_\nabla$ and $h_\nabla ^{(j)}$ are smooth on $(0, \infty) \times M \times M$, and since $C$ is dense in $M$ by virtue of being co-null, an elementary argument shows that $C=M$.
\end{proof}

\end{document}